\documentclass[12pt]{amsart}
\usepackage[utf8]{inputenc} 
\usepackage[active]{srcltx}
\usepackage{a4wide}
\usepackage{amsthm,amsfonts,amsmath,mathrsfs,amssymb}
\usepackage{dsfont}
\usepackage{mathtools}
\usepackage[T1]{fontenc}
\usepackage[utf8]{inputenc}
\usepackage{enumerate}
\usepackage{comment}
\usepackage[left=4cm,top=4cm,right=4cm,bottom=4cm]{geometry}
\numberwithin{equation}{section}

\usepackage{geometry}
\usepackage{graphicx}
\usepackage{amssymb}
\usepackage{amsmath}
\usepackage{amsthm}
\usepackage{listings}
\usepackage[colorlinks=true,urlcolor=blue,citecolor=blue,linkcolor=blue]{hyperref}
\usepackage{esint}
\usepackage{xcolor}

\usepackage{etex}
\usepackage[all]{xy}
\usepackage{pgf,tikz}
\usetikzlibrary{arrows}
\usetikzlibrary{patterns}
\usepackage{hyperref}
       \def\t{\theta}

\newcommand{\R}{{\mathbb R}} 
\newcommand{\C}{{\mathbb C}} 
\newcommand{\N}{{\mathbb N}} 
 
\newcommand{\Q}{{\mathbb Q}} 
\newcommand{\T}{{\mathbb T}} 
\newcommand{\D}{{\mathbb D}}

\newtheorem{teor}{Theorem}[section]

\newtheorem{lema}[teor]{Lemma}
\newtheorem{prop}[teor]{Proposition}
\newtheorem{defi}[teor]{Definition}
\newtheorem{example}[teor]{Example}

\theoremstyle{remark}
\newtheorem{remark}[teor]{Remark}
\DeclareMathOperator*{\esssup}{ess\,sup}
\makeatletter
\@namedef{subjclassname@2020}{%
  \textup{2020} Mathematics Subject Classification}
\makeatother

\title{Volterra operator acting on Bergman spaces of Dirichlet series}

\author[C. G\'omez-Cabello ]{Carlos G\'omez-Cabello}
\address{Departamento de Matem\'atica Aplicada II and IMUS, Escuela Politécnica Superior, Universidad de Sevilla, Calle Virgen de África, 7 41011 Sevilla, Spain}
\email{cgcabello@us.es}

\author[P. Lefevre]{Pascal Lefèvre}
\address{
Univ. Artois, UR 2462, Laboratoire de Mathématiques de Lens (LML), F-62 300 LENS,
FRANCE
}
\email{pascal.lefevre@univ-artois.fr}

\author[H. Queff\'elec]{Herv\'e Queff\'elec}
\address{Univ. Lille Nord de France, USTL, Laboratoire Paul Painlevé U.M.R. CNRS 8524, F-59 655
VILLENEUVE D’ASCQ Cedex, FRANCE
}
\email{herve.queffelec@univ-lille.fr
}

\date{\today}
\begin{document}

\maketitle 
\begin{abstract}
	Since their introduction in 1997, the Hardy spaces of Dirichlet series have been broadly and deeply studied. The increasing interest sparked by these Banach spaces of Dirichlet series motivated the introduction of new such spaces, as the Bergman spaces of Dirichlet series $\mathcal{A}^p_{\mu}$ here considered, where $\mu$ is a probability measure on $(0,\infty)$. Similarly, recent lines of research have focused their attention on the study of some classical operators acting on these spaces, as it is the case of the Volterra operator $T_g$. In this work, we introduce a new family of Bloch spaces of Dirichlet series, the $\text{Bloch}_{\mu}$-spaces, and study some of its most essential properties. Using these spaces we are able to provide a sufficient condition for the Volterra operator $T_g$ to act boundedly on the Bergman spaces $\mathcal{A}^p_{\mu}$. We also establish a necessary condition for a specific choice of the probability measures $\mu$. Sufficient and necessary conditions for compactness are also proven. The membership in Schatten classes is studied as well. Eventually, a radicality result is established for Bloch spaces of Dirichlet series.
\end{abstract}
\tableofcontents

\section{Introduction}
Since the introduction of the first Banach spaces of Dirichlet series in 1997 by  Hedenmalm, Lindqvist, and Seip \cite{HedLinSeip}, the interest for these spaces has done nothing but increase. In their work, Hedenmalm et al. introduced the first Hardy spaces in the setting of Dirichlet series, namely, the spaces $\mathcal{H}^2$ and $\mathcal{H}^{\infty}$. Later on, Bayart \cite{bayart1} extended the definition of the spaces to the range $1\leq p<\infty$. To do so, he defined the spaces as the completion of the Dirichlet polynomials on the $L^p(\T^{\infty})$ norm, where $\T^{\infty}$ denotes the infinite polytorus.

More recently, in \cite{pascal}, Bailleul and Lefèvre introduced another family of Banach spaces of Dirichlet series  and studied some of their most essential properties. More specifically, two different classes of Bergman spaces of Dirichlet series were defined. The first one was the Bergman spaces $\mathcal{A}^p_{\mu}$, $1\leq p<\infty$, with $\mu$ a probability measure on $(0,\infty)$ such that $0\in\text{supp}\{\mu\}$. The second class was $\mathcal{B}^p$, $1\leq p<\infty$, defined in a similar fashion to the $\mathcal{H}^p$-spaces but from a closure process in $L^p(\D^{\infty})$ instead of $L^p(\T^{\infty})$, where $\D^{\infty}$ denotes the infinite polydisc. The interesting thing about these two classes of spaces is that when they are defined in finite dimension, this is, whenever the Bohr lift to the infinite polydisc $\D^{\infty}$ depends on finitely many variables, the spaces coincide. Nonetheless, when working with infinitely many variables, these two classes of spaces show a rather different behaviour. In particular, this can be seen in the study of the boundedness of some relevant operators.

As in the case of the classical Banach spaces of analytic functions in the unit disc, in the Dirichlet series setting, some of the most remarkable operators subject study are the composition operators (\cite{gorheda}, \cite{bayart2}, \cite{maxim}, \cite{maximbrevig}) or the Volterra operators. For a Dirichlet series $g(s)=\sum_{n=1}^\infty a_nn^{-s}$, the Volterra operator of symbol $g$, denoted by $T_g$, acting on a convergent Dirichlet series $f$ is defined as
\begin{equation}\label{volterraop}
    T_gf(s):=-\int_s^{\infty}f(w)g'(w)dw.
\end{equation}
A first remarkable difference between the operator $T_g$ when considered in the setting of Banach spaces of Dirichlet series with respect to the unit disc classical spaces is that the integration operator cannot be embedded in the study of the operator $T_g$. Indeed, in the unit disc setting, considering as symbol $g$ the identity map, $T_g$ is essentially the integration operator. 
Nonetheless, since the identity map is not a Dirichlet series, these operators must be studied separately when dealing with Dirichlet series. For a real number $\theta$, we shall denote $\C_{\theta}=\{s\in\C:\Re(s)>\theta\}$. For $\theta=0$, we write $\C_+=\C_0$.

\ 
The boundedness of the operator $T_g$ is described through the membership of the \emph{symbol} $g$ to a certain space of functions $X$. To the best of the authors' knowledge, the first systematic study of the boundedness of the operator $T_g$ acting on Banach spaces of Dirichlet series was carried out in \cite{brevig} by Brevig, Perfekt, and Seip. There, a sufficient condition for boundedness on $\mathcal{H}^p$ was given. Namely, the membership of $g$ in $\text{BMOA}(\C_+)$. It was also established that whenever $T_g$ is bounded on $\mathcal{H}^p$ and, surprisingly, $p\in\Q^+$, then $T_g$ belongs to the space $\text{BMOA}(\C_{1/2})$. Moreover, it was shown that the sufficient condition fails to be necessary and viceversa.

\ 
After this seminal work, several papers studying the operator $T_g$ acting on other Banach spaces of Dirichlet series have appeared. In \cite{bommier}, Bommier-Hato gave a sufficient and a necessary condition for the boundedness of $T_g$ on the Hilbert version of a different family of weighted Bergman spaces of Dirichlet series, namely, the spaces $\mathcal{H}^p_w$. Bommier-Hato showed that the membership of $g$ to the Bloch space $\text{Bloch}(\C_+)$ is sufficient to guarantee the boundedness of $T_g$ on $\mathcal{H}^2_w$, as long as its Bohr lift depends on finitely many variables. In the same work, she also gave a necessary condition for boundedness in $\mathcal{H}^2_w$ in terms of the space $\text{Bloch}(\C_{1/2})$ of Dirichlet series. Regarding the sufficiency, very recently, Chen and Wang in \cite{chen} were able to extend the range of $p$ to $1\leq p<\infty$ as well as remove the previous restriction on the Bohr lift of the symbol. However, they either restricted themselves to the smaller space $\text{BMOA}(\C_+)$, being the membership to this space sufficient for boundedness; either they required 
the symbol $g$ to be $1$-homogeneous, that is
\[
g(s)=\sum_{\text{$p$ prime}}a_pp^{-s}.
\]
In the second case, they showed that $T_g$ is bounded on $\mathcal{H}^p_w$ if and only if $g\in\mathcal{H}^2_w$.

%- Paper identidad L-P.
Regarding the spaces $\mathcal{A}^p_{\mu}$, the first study of the operator $T_g$ was tackled in \cite{prisa}. More precisely, Fu, Guo, and Yan considered a specific family of measures $\mu_{\alpha}$, $\alpha>0$ and gave a sufficient condition for the boundedness of $T_g$ when acting on $\mathcal{A}^p_{\alpha}=\mathcal{A}^p_{\mu_{\alpha}}$, $1\leq p<\infty$. They also provide a necessary condition for the boundedness of $T_g$ acting on $\mathcal{A}^2_{\alpha}$. Let us point out that in the work \cite{prisa} the authors consider the family of measures
\[
d\nu_{\alpha}(\sigma)=\frac{2^{\alpha-1}}{\Gamma(\alpha-1)}\sigma^{\alpha-2}d\sigma,\quad \alpha>1, \ \sigma>0.
\]
This is, they get rid of the exponential term in the definition of the measures $\mu_{\alpha}$ from \eqref{macaden1}. Observe that the measures $d\nu_{\alpha}$ are not probability measures, but we can still define in an identical manner the spaces $\mathcal{A}^p_{\nu_{\alpha}}$. In fact, the resulting $\mathcal{A}^p_{\nu_{\alpha}}$-norm and the $\mathcal{A}^p_{\mu_{\alpha}}$-norm happen to be equivalent (see \cite[Lemma 3.9]{GLQ2}); giving the same $\mathcal{A}^p_{\mu}$ spaces as a result of the \emph{a priori} different completion processes. Because of this, all the results obtained in this work are also valid for the family of measures considered in \cite{prisa}.

\ 

The organisation of the paper is as follows. 
\begin{itemize}
    \item In Section \ref{sec:bloch} we introduce a new family of Bloch spaces of Dirichlet series and present some of its most essential properties.
    \item  In Section \ref{sec:volterrainbloch} we characterise the symbols giving rise to bounded Volterra operators on the spaces $\text{Bloch}(\C_+)$ of Dirichlet series. This will allow us to introduce some ideas that will appear later in the  proof of the necessary condition. It will also show how different the difficulty between the characterisation of the boundedness in these spaces and the spaces $\mathcal{A}^p_{\mu}$ is.
    \item The proof of the sufficient and the necessary condition, as in the $\mathcal{H}^p$-spaces, requires %of 
    integral representations of the norm. These representations are given by means of Littlewood-Paley type results. This is the content of Section \ref{section:LP}. 
    \item  In Section \ref{sec:volterraonberg} we give a sufficient condition for the boundedness of $T_g$ on $\mathcal{A}^p_{\mu}$ depending on the measure $\mu$. More precisely, it will be sufficient for the symbol to belong to the Bloch-type space $\text{Bloch}_{\mu}(\C_+)$. In particular, the sufficient condition given in \cite{prisa} is contained in our work as a particular choice of the measure $\mu$. We also provide a necessary condition for the special case $\mathcal{A}^p_{\alpha}$, where $\alpha>-1$ and $p\in[1,\infty)$, thus, improving the already known results from \cite{prisa} by extending both the range of $p$ and $\alpha$. The proof is mainly based on the estimate of the norm of the pointwise evaluation functionals in $\mathcal{A}^p_{\alpha}$ from our recent paper \cite[Theorem 5.2]{GLQ2}. Necessary and sufficient conditions for compactness are given too, as well as a characterisation of the membership in the Schatten classes.
    \item In the last section, we provide a radicality result inspired by the recent work \cite{alem}.
    \item Section \ref{sec:dirichlet} and Section \ref{sec:apmuetal} are expository and they are devoted to recall all the results and definitions which will be needed in subsequent sections.
\end{itemize}

\ 

From now on, unless indicated otherwise, $\mu$ will stand for a probability measure on $(0,+\infty)$ such that $0\in\text{supp}(\mu)$, where $\text{supp}(\mu)$ denotes the support of the measure $\mu$. In order to shorten the statements of some results throughout the paper, we shall say that $\mu$ is an \emph{admissible measure} whenever it satisfies the latter requirements.

\ 

We will use the notation $f(x)\lesssim g(x)$ if there is some constant $C>0$ such that $|f(x)|\leq C|g(x)|$ for all $x$. If we have simultaneously that $f(x)\lesssim g(x)$ and $g(x)\lesssim f(x)$, we write $f\approx g$.

\

\noindent
\textbf{Acknowledgements.} The authors would like to thank Professor Manuel D. Contreras and Professor Luis Rodríguez Piazza for carefully reading the paper, as well as for their many valuable and insightful remarks, which have notably contributed to improve the exposition of the results and their statements.

\section{Dirichlet series}\label{sec:dirichlet}
A convergent Dirichlet series $\varphi$ is a formal series
$$
\varphi(s)=\sum_{n=1}^{\infty}a_nn^{-s},
$$
which converges somewhere, equivalently, in some half-plane $\C_{\theta}$ (see \cite[Lemma 4.1.1]{queffelecs}) and we write $\varphi\in\mathcal{D}$. 
When a Dirichlet series  $\varphi$ has only finitely many non-zero coefficients, we say that $\varphi$ is a \emph{Dirichlet polynomial}. If $\varphi$ just has one non-zero coefficient, the series $\varphi$ is a \emph{Dirichlet monomial}. 
The first remarkable difference between Taylor series and Dirichlet series is that the latter converge in half-planes. 
In fact, there exist several \emph{abscissae of convergence} which, in general, do not coincide. The most remarkable ones are the following:  
$$
\sigma_{c}(\varphi)=\inf\{ \sigma:  \sum_{n=1}^{\infty}a_nn^{-s} \textrm{ is convergent on $\C_{\sigma}$}\};
$$
$$
\sigma_{u}(\varphi)=\inf\{ \sigma:  \sum_{n=1}^{\infty}a_nn^{-s} \textrm{ is uniformly convergent on $\C_{\sigma}$}\};
$$
$$
\sigma_{b}(\varphi)=\inf\{ \sigma:  \sum_{n=1}^{\infty}a_nn^{-s} \textrm{ has a bounded analytic extension to $\C_{\sigma}$} \};
$$
$$
\sigma_{a}(\varphi)=\inf\{ \sigma:  \sum_{n=1}^{\infty}a_nn^{-s} \textrm{ is absolutely convergent on $\C_{\sigma}$}\}.
$$

Let us mention that $\sigma_u=\sigma_b$ by Bohr's theorem (see \cite[Th.6.2.3]{queffelecs}).

The infinite polytorus $\T^{\infty}$ can be identified with the group of complex-valued characters $\chi$ on positive integers, more precisely, completely multiplicative maps satisfying $|\chi(n)|=1$ for all $n\in\N$ and $\chi(mn)=\chi(m)\chi(n)$, for all $m,n\in\N$. Then, given $\chi\in\T^{\infty}$ and a Dirichlet series $f(s)=\sum_{n\geq1}a_nn^{-s}$ converging in some half-plane $\C_{\theta}$, the Dirichlet series
\[
f_{\chi}(s)=\sum_{n=1}^{\infty}a_n\chi(n)n^{-s}.
\]
is the normal limit in $\C_{\theta}$ of some sequence $\{f_{i\tau_k}\}_k$, where $\{\tau_k\}\in\R$.

We recall that, when $f$ is a Dirichlet series converging in some half place $\C_a$, then, for $s_0\in\overline{\C_0}$, we denote
\begin{equation*}
\forall s\in\C_a\,,\qquad f_{s_0}(s)=f(s+s_0).
\end{equation*}

\

Given a convergent Dirichlet series $f$ and $\sigma>0$, we denote by $f_{\sigma}$ %to 
the convergent Dirichlet series resulting from translating $f$ by $\sigma$: $f_{\sigma}(s)=f(\sigma+s)$. 
It is worth noting that the horizontal translation can be regarded as an operator acting on convergent Dirichlet series. In this case, we denote it by $T_{\sigma}$ so that $T_{\sigma}f=f_{\sigma}$.  
It is noteworthy to observe that the $\mathcal{H}^p$ spaces remain stable under horizontal translations to the right or vertical translations, see for instance \cite[Theorem 11.20]{defant-peris}.

\section{A family of Bloch spaces of Dirichlet series}\label{sec:bloch}
The following Bloch spaces were introduced in \cite{bommier}. 
\begin{defi}\label{blochcl}
Let $\theta\geq0$. An analytic function $f$ in the half-plane $\C_{\theta}$ belongs to the Bloch space $\text{Bloch}(\C_{\theta})$ if
\begin{equation}
    \sup_{\sigma+it\in\C_{\theta}}(\sigma-\theta)|f'(\sigma+it)|<\infty.
\end{equation}
\end{defi}
These spaces are Banach spaces of analytic functions with the norm
\[
\|f\|_{\text{Bloch}(\C_{\theta})}:=\sup_{\sigma+it\in\C_{\theta}}(\sigma-\theta)|f'(\sigma+it)|+|f(1)|.
\]
In particular, the space $\text{Bloch}(\C_{\theta})\cap\mathcal{D}$ is a closed subspace of $\text{Bloch}(\C_{\theta})$.
\begin{remark}\label{sismuertos}
   Observe that,  by Cauchy's inequality, we have that $H^{\infty}(\C_{+})$, the space of bounded analytic functions in $\C_{+}$, is a subspace of $\text{Bloch}(\C_{+})$. In particular,
   \[
   \mathcal{H}^{\infty}\subset \text{Bloch}(\C_{+})\cap\mathcal{D}.
   \]
   Obviously, this is true for every half-plane $\C_{\theta}$, $\theta\geq0$.
\end{remark}
\begin{remark}\label{concejo}
    For $\theta\geq0$, we also have that $\text{Bloch}(\C_{\theta})\cap\mathcal{D}\subset\mathcal{H}^{\infty}(\C_{\theta+\varepsilon})$, for every $\varepsilon>0$. Let us point out that this was already observed for the case $\theta=0$ in \cite[Lemma 3]{bommier}. We prove it for $\theta=0$, being the argument being identical for the general case. Let $\varepsilon>0$. Since $f\in\mathcal{D}$, there exists $M>0$ such that $f\in\mathcal{H}^{\infty}(\C_M)$. If $M\leq\varepsilon$, we are done. Hence, we suppose $M>\varepsilon$. Take $s\in\C_{\varepsilon}$, then
    \begin{align*}
        |f(s)|&\leq |f(s+M-\varepsilon)|+\int_{[s,s+M-\varepsilon]}|f'(z)||dz|\\&
        \leq
        \|f\|_{\mathcal{H}^{\infty}(\C_M)}+(M-\varepsilon)\frac{A}{\varepsilon}<\infty.
    \end{align*}
    Therefore, $f$ is bounded in every half-plane $\C_{\varepsilon}$, $\varepsilon>0$, as desired.
\end{remark}

Since the spaces were introduced in relation to the boundedness of the Volterra operator $T_g$ on some Hilbert spaces of Dirichlet series, the most remarkable cases are $\theta=0$ and $\theta=1/2$. It is also worth mentioning that the space $\text{Bloch}(\C_+)$ was treated with considerable detail in \cite{queffetal}.

\begin{defi}
   Let $\omega\colon(0,1]\to(0,\infty)$ be a measurable map such that 
   \begin{equation}\label{condicionomega}
       \omega\in L^{\infty}([\varepsilon,1])\quad \text{and}\quad 1/\omega\in L^{\infty}([\varepsilon,1]), \quad \text{for all $\varepsilon>0$}.
   \end{equation}
 We define the $\text{Bloch}_{\omega}(\C_+)$ space as the collection of analytic functions $f$ in $\C_+$ such that
   \[
   \esssup_{\substack{0<\sigma\leq1\\ t\in\R}}\omega(\sigma)|f'(\sigma+it)|+\|f\|_{H^{\infty}(\C_1)}<\infty.
   \]
\end{defi}
\begin{remark}
    The first condition in \eqref{condicionomega} ensures that the spaces $\text{Bloch}_{\omega}(\C_+)$ have an interesting structure since, otherwise, the spaces merely consist of constant functions.  
\end{remark}
\begin{remark}
 The second condition in \eqref{condicionomega} guarantees that the norm convergence implies local uniform convergence. This ensures that these spaces $\text{Bloch}_{\omega}(\C_+)$ are Banach spaces of analytic functions with the norm
\[
\|f\|_{\text{Bloch}_{\omega}(\C_+)}:=\esssup_{\substack{0<\sigma\leq1\\ t\in\R}}\omega(\sigma)|f'(\sigma+it)|+\|f\|_{H^{\infty}(\C_1)}.
\]
Clearly, we also have that $\text{Bloch}_{\omega}(\C_+)\cap\mathcal{D}$ is a closed subspace of $\text{Bloch}_{\omega}(\C_+)$.
\end{remark}

\begin{remark}
    In the case $\omega(\sigma)=\sigma$, we have that
    \[
    \text{Bloch}_{\omega}(\C_+)\subsetneq \text{Bloch}(\C_+).
    \]
    Let us see this. Consider $f\in\text{Bloch}_{\omega}(\C_+)$. Then, 
    \[
    \sup_{\substack{0<\sigma\leq1\\ t\in\R}}\sigma|f'(\sigma+it)|+\|f\|_{H^{\infty}(\C_1)}<\infty.
    \]
       Hence, by definition, $f\in H^{\infty}(\C_1)$. Observe also that, by the same argument from Remark \ref{concejo}, $f\in H^{\infty}(\C_{1/2})$. Then, by Cauchy's inequality, we have that 
    \[
    (\sigma-1/2)|f'(\sigma+it)|\leq M,\quad \text{for all $\sigma>1/2$ and $t\in\R$}.
    \]
    Then, 
    \[
    \frac{\sigma}{2}|f'(\sigma+it)|\leq M,\quad \text{for all $\sigma>1$ and $t\in\R$}.
    \]
     Putting all together, we obtain
    \[
     \sup_{\substack{\sigma>0\\ t\in\R}}\sigma|f'(\sigma+it)|<\infty,
    \]
    as desired. For the strict inclusion, it suffices to consider the function $f(z)=\log z$, $z\in\C_+$. Clearly,
    \[
    \sigma|f'(\sigma+it)|\leq1.
    \]
    However, $f$ is not bounded in any half-plane $\C_{\varepsilon}$, $\varepsilon>0$.
\end{remark}
\begin{remark}\label{esmuyutil}
 For the special choice $\omega(\sigma)=\sigma$, we have that
  \[
  \text{Bloch}_{\omega}(\C_+)\cap\mathcal{D}=\text{Bloch}(\C_+)\cap\mathcal{D}.
  \]
 The fact that $\text{Bloch}_{\omega}(\C_+)\cap\mathcal{D}\subset\text{Bloch}(\C_+)\cap\mathcal{D}$ follows immediately from the previous remark. For the converse inclusion, let $f\in\text{Bloch}(\C_+)\cap\mathcal{D}$. Remark \ref{concejo} gives the membership of $f$ to $H^{\infty}(\C_1)$. Regarding the first term of the norm $\text{Bloch}_{\omega}(\C_+)$, it is clearly finite since
  \[
\sup_{\substack{0<\sigma\leq1\\t\in\R}}\sigma|f'(\sigma+it)|\leq \sup_{\substack{0<\sigma\\t\in\R}}\sigma|f'(\sigma+it)|.
  \]
  \end{remark}
\begin{remark}
   Working as in Remark \ref{concejo}, we have that whenever $1/\omega$ is in $L^1(0,1)$, the resulting space $\text{Bloch}_{\omega}(\C_+)$ is contained in $H^{\infty}(\C_+)$, even in $A(\C_+)$, the algebra of bounded analytic functions in $\C_+$ which are also uniformly continuous in the same half-plane. 
    \end{remark}
Since we are interested in the $\text{Bloch}_{\omega}$ spaces because of their connection with the operator $T_g$ acting on the spaces $\mathcal{A}^p_{\mu}$, we shall pay special attention to a family of Bloch spaces where the expression of the weight $\omega$ is intimately connected to the measure $\mu$.

\ 
Let $h$ be a positive function such that $\int_0^{\varepsilon}h(\sigma)d\sigma>0$ for all $\varepsilon>0$ and $h\in L^1(0,\infty)$. In \cite{pascal}, Bailleul and Lefèvre introduced the following function associated to this density
\[
\beta_{h}(\sigma):=\int_0^{\sigma}(\sigma-u)h(u)du.
\]
However, it is possible to define the function $\beta_{\mu}$ for any admissible measure $\mu$ %\textrm{i.e.} any probability measure $\mu$ on $(0,+\infty)$ such that $0\in\text{supp}(\mu)$. 
Indeed, in this case
\begin{equation}\label{betamu}
   \beta_{\mu}(\sigma):=
\int_0^{\sigma}(\sigma-u)d\mu(u).
\end{equation}
Notice that $\beta_h$ is nothing but writing $\beta_{\mu}$ when $d\mu(u)=h(u)du$. We may sometimes also write $\beta_{\mu}$ alternatively as:
\[
\beta_{\mu}(\sigma)=\int_0^{\sigma}\int_0^td\mu(u)dt
=\int_0^{\sigma}\mu([0,t])dt.
\]
To see this, it suffices to apply Fubini in \eqref{betamu}. Observe also that $\beta_{\mu}$ is non-decreasing on $(0,\infty)$. Moreover, since $\sigma\mapsto\int_0^{\sigma}d\mu(u)$ is non-decreasing too, we have that $\beta_{\mu}$ is a convex function on $(0,\infty)$. Furthermore, $\beta_{\mu}(\sigma)\leq\sigma$, which implies that $\lim_{\sigma\to\infty}\beta_{\mu}(\sigma)n^{-2\sigma}=0$ for all $n\geq2$. 

\ 

For our purposes, we shall write the probability measure $\mu$ as $$d\mu(\sigma)=h(\sigma)d\sigma+d\mu_s(\sigma),$$ where $h\in L^1(d\sigma)$ and $\mu_s$ stands for its singular part (see \cite[Theorem 6.10]{rudin}). 
\begin{defi}\label{mubloch}
 Let $\mu$ be an admissible measure and $\beta_{\mu}$ as in \eqref{betamu}. Let $h$ be as before and such that
 \begin{equation}\label{condicionh}
    h\in L^{\infty}([\varepsilon,1])\quad \text{and}\quad 1/h\in L^{\infty}([\varepsilon,1]), \quad \text{for all $\varepsilon>0$}.
 \end{equation} 
 We say that an analytic function $f$ in the half-plane $\C_+$ belongs to the Bloch space $\text{Bloch}_
    {\mu}(\C_+)$ if
    \begin{equation*}
        \|f\|_{\text{Bloch}_{\mu}(\C_+)}=\esssup_{\substack{0<\sigma\leq1 \\ t\in\R}}\omega(\sigma)|f'(\sigma+it)|+
        \|f\|_{H^{\infty}(\C_1)}
        <\infty,
    \end{equation*}
    where 
    \begin{equation}\label{omega}
        \omega(\sigma)=\sqrt{\frac{\beta_{\mu}(\sigma)}{h(\sigma)}}\cdot
    \end{equation}
\end{defi}
\begin{remark}
Observe that, by the definition of $\omega$ in \eqref{omega} and the function $\beta_{\mu}$ being continuous and bounded on $(0,1]$, the conditions in \eqref{condicionh} are equivalent to the ones in \eqref{condicionomega}.
\end{remark}
From now on, we shall work with the spaces $\text{Bloch}_{\mu}$, meaning that the subindex appearing in the $\text{Bloch}$ spaces will stand for the specific weight $\omega$ from \eqref{omega}, and not for a general weight.

\ 

The interest in defining these Bloch spaces in such a way arises from its connection to the boundedness of the Volterra operator $T_g$ on $\mathcal{A}^p_{\mu}$. As we shall see, we establish a sufficient condition depending on the measure $\mu$ defining $\mathcal{A}^{p}_{\mu}$ in terms of these Bloch-type spaces. To the best of the author's knowledge, these spaces are new in the literature.
\begin{remark}
The fact that we require $h$ to be non-zero almost everywhere in $(0,1]$ guarantees the correct definition of $\omega$.
\end{remark}
\begin{remark}
It is also worth mentioning why we restrict ourselves to the strip $\mathbb{S}=\{s:0<\text{Re}(s)<1\}$ instead of taking it in the whole real line as it is the case for the space $\text{Bloch}(\C_+)$. 
Indeed, if we take $d\mu=2e^{-2\sigma}d\sigma$, we can easily check that $\omega(\sigma)\approx\sigma$ in $(0,1)$ and $\omega(\sigma)\approx \sigma^{1/2}e^{\sigma}$ when $\sigma\to+\infty$. 
If the supremum were to be taken on the whole right half-plane, then the monomial $2^{-s}$ would fail to belong to this $\text{Bloch}_{\mu}$ space. 
To avoid this inconvenience, we restrict the supremum to the vertical strip $\mathbb{S}$.
\end{remark}
\begin{remark}\label{eslomismo}
If we consider the family of measures $\mu_{\alpha}$ given by
    \begin{equation}\label{macaden1}
        d\mu_{\alpha}(\sigma)=\frac{2^{\alpha+1}}{\Gamma(\alpha+1)}\sigma^{\alpha}e^{-2\sigma}\,d\sigma,\quad \alpha>-1,
\end{equation}
we have that in a neighbourhood of $0$, $\beta_{\mu_{\alpha}}(\sigma)\approx \sigma^{\alpha+2}$, as $\sigma\to0^+$. Therefore, in $(0,1)$,
\begin{align*}
    \omega(\sigma)=\sqrt{\frac{\beta_{\mu_{\alpha}}(\sigma)}{h(\sigma)}}\approx
    \sqrt{\sigma^{2}e^{2\sigma}}\approx \sigma.
\end{align*}
Hence, in this case, by Remark \ref{esmuyutil}, the associated Bloch$_{\mu_{\alpha}}(\C_+)\cap\mathcal{D}$ space coincides with the Bloch space $\text{Bloch}(\C_+)\cap\mathcal{D}$. 
\end{remark}
Naturally, we can define the $\text{Bloch}_{\mu}$ spaces in any half-plane $\C_{\theta}$, $\theta>0$, by requiring
\[
\sup_{\substack{\theta<\sigma\leq\theta+1\\t\in\R}}\omega(\sigma-\theta)|f'(\sigma+it)|+\|f\|_{H^{\infty}(\C_{\theta+1})}<\infty,
\]
where $\omega$ is as in Definition \ref{mubloch}.

\ 

As it was mentioned before, the motivation behind the introduction of such spaces is its connection to the boundedness of the Volterra operator on the Banach spaces of Dirichlet series $\mathcal{A}^p_{\mu}$. As we are about to see, Dirichlet series belonging to the spaces $\text{Bloch}_{\mu}(\C_+)$ have some good convergence properties.
\begin{lema}\label{born}
    Let $f$ be in $\emph{Bloch}_{\mu}(\C_+)$ and $\varepsilon>0$. Then, given $\varepsilon<\sigma\leq1$ and $t\in\R$,
      \begin{equation*}        |f(\sigma+it)|\lesssim\left(\frac{1}{(\beta_{\mu}(\varepsilon))^{1/2}}+1\right)\|f\|_{{\emph{Bloch}_{\mu}(\C_+)}}.        \end{equation*}
      Moreover, if $f\in\mathcal{D}$, then  $\sigma_b(f)=\sigma_b(f')\leq 0$.
\end{lema}
\begin{proof}
    Let $\varepsilon>0$ and $s=\sigma+it$, such that $\varepsilon<\sigma\leq1$. Then, by the definition of the $\text{Bloch}_{\mu}(\C_{+})$ space, we have, for $t\in\R$,
    \begin{align*}
    |f(\sigma+it)|&\leq|f(\sigma+it)-f(1+it)|+|f(1+it)|\\
    &\leq
    \int_{\sigma}^1|f'(x+it)|dx+\|f\|_{\mathcal{H}^{\infty}(\C_1)}
    \\
    &\leq\sup_{\substack{0<u<1\\ t\in\R}}\omega(u)|f'(u+it)|\int_{\sigma}^1\frac{dx}{\omega(x)}+\|f\|_{\mathcal{H}^{\infty}(\C_1)}\\
    &\leq \left(
\int_{\sigma}^1\frac{dx}{\omega(x)}+1
    \right)\|f\|_{\text{Bloch}_{\mu}(\C_+)}.
\end{align*}
Now, using the fact that $\beta_{\mu}$ is non-decreasing and $\int_0^1h(\sigma)d\sigma\leq1$, we obtain
\begin{align*}
 \left(\int_{\sigma}^1\frac{dx}{\omega(x)}\right)^2\leq\int_{\sigma}^1\frac{dx}{\omega^2(x)}=   \int_{\sigma}^1\frac{h(x)}{\beta_{\mu}(x)}dx
    \leq 
    \frac{1}{\beta_{\mu}(\sigma)}\int_{\sigma}^1h(x)dx\leq
    \frac{1}{\beta_{\mu}(\varepsilon)}\cdot
\end{align*}
Therefore,
\[
|f(\sigma+it)|\lesssim\left(\frac{1}{(\beta_{\mu}(\varepsilon))^{1/2}}+1\right)\|f\|_{\text{Bloch}_{\mu}(\C_{+})},\quad \sigma>\varepsilon.
\]
If $f\in\mathcal{D}$, by the very definition of $\sigma_b$, we have that $\sigma_b(f)\leq0$. The equality $\sigma_b(f)=\sigma_b(f')$ holds for every $f\in\mathcal{D}$, thanks to Cauchy integral formula.
\end{proof}

The $\text{Bloch}_{\mu}(\C_{\theta})$ spaces, $\theta\geq0$, satisfy the expected inclusion relations.
\begin{lema}\label{inclusion}
    Let $\theta>\eta$. Then, $\emph{Bloch}_{\mu}(\C_{\eta})\subset \emph{Bloch}_{\mu}(\C_{\theta})$ and, in fact, given $f\in\emph{Bloch}_{\mu}(\C_{\eta})$, $$\|f\|_{\emph{Bloch}_{\mu}(\C_{\theta})}\leq \|f\|_{\emph{Bloch}_{\mu}(\C_{\eta})}.$$
\end{lema}
 \begin{proof}
Let $\varepsilon>0$ and $f\in H^{\infty}(\C_{\varepsilon})$. We define
\[
\alpha(\sigma):=\sup_{t\in\R}|f'(\sigma+it)|,\quad\sigma>\varepsilon.
\]
 Then, we claim that $\alpha$ is non-increasing on $(\varepsilon,\infty)$ and $\lim_{\sigma\to+\infty}\alpha(\sigma)=0$. The first statement is a direct consequence of the Maximum Modulus Principle on half-planes (see \cite[Chapter 12, Exercise 9]{rudin}). Regarding the second part of the claim, by the definition of $\alpha$ and Cauchy's inequality, we have that, 
 \[
 \alpha(\sigma+a)\leq \|f'\|_{H^{\infty}(\C_{\varepsilon+a})}\leq \frac{\|f\|_{H^{\infty}(\C_{\varepsilon})}}{a},\quad a>0,
 \]
 and the claim follows.

\
 For simplicity, we shall consider $\eta=0$. The proof can be easily adapted for general $\eta$. Let $f$ belong to $\text{Bloch}_{\mu}(\C_+)$,  
\begin{align*}
\sup_{\substack{\theta<\sigma\leq 1+\theta\\ t\in\R}}
\omega(\sigma-\theta)|f'(\sigma+it)|
&=
\sup_{\substack{0<\varepsilon\leq 1\\ t\in\R}}\omega(\varepsilon)|f'(\theta+\varepsilon+it)|\\
&=\sup_{0<\varepsilon\leq1}\omega(\varepsilon)\left(
\sup_{t\in\R}|f'(\theta+\varepsilon+it)|
\right)\\
&\leq\sup_{0<\varepsilon\leq1}\sup_{t\in\R}\omega(\varepsilon)|f'(\varepsilon+it)|
,
\end{align*}
where in the inequality we have used the claim. The conclusion follows using the above estimate and the fact that  $\|f\|_{H^{\infty}(\C_{\theta+1})}\leq \|f\|_{H^{\infty}(\C_1)}$ for $\theta\geq0$.
\end{proof}
In the next lemma, we are interested in a simple criterion for the membership in the space of Bloch spaces associated to a weight of potential type given in terms of the coefficients of the series. We provide some examples of measures inducing such weights in Example \ref{ejemplomeasure}. This result is inspired by \cite{queffetal}, where a similar result is given for the space $\text{Bloch}(\C_+)$.

\begin{lema}\label{belongcoef}
Let $\mu$ be such that $\omega(\sigma)\approx\sigma^{\delta}$ on $(0,1)$, $\delta>0$. Let $f(s)=\sum_{n\geq1}a_nn^{-s}$ be a Dirichlet series bounded in $\C_1$, with non negative real coefficients. Then, $f\in\emph{Bloch}_{\mu}(\C_+)$ if and only if, for $x\geq\textrm{e}$, 
\begin{equation}\label{condan}
\sum_{x\leq n\leq x^2}a_n=O((\log x)^{\delta-1}).
\end{equation}
\end{lema}
\begin{proof}
Let $f(s)=\sum_{n\geq1}a_nn^{-s}$ be a Dirichlet series in $\text{Bloch}_{\mu}(\C_+)$. Then,
	\[
	\sup_{\sigma\in(0,1)}\left(
	\sum_{n=1}^{\infty}\sigma^{\delta}\big(\log n\big)a_ne^{-\sigma \log n}
	\right)=c<\infty.
	\]
It is clear that, given $0<\sigma<1$
	\[
	c\geq\sum_{x\leq n\leq x^2}\sigma^{\delta}\big(\log n\big)a_ne^{-\sigma \log n}
	\geq 
	\Big(\sum_{x\leq n\leq x^2}a_n\Big)\sigma^{\delta-1}(\sigma\log x)e^{-2\sigma \log x}\,.
	\]
The choice of $\sigma=1/\log(x)$ gives condition \eqref{condan}.
	
	\ 
	Now, assume that condition \eqref{condan} holds. Consider $x_j=\exp(2^j/\sigma)$. Clearly, $x_{j+1}=x_j^2$, so that
\begin{align*}
	\sum_{n=1}^{\infty}\sigma^{\delta}\big(\log n\big)a_ne^{-\sigma \log n}
	&\leq\sum_{j}\sum_{x_j\leq n \leq x_{j+1}}\sigma^{\delta}\big(\log n\big)a_ne^{-\sigma \log n}
	\\& \leq
	\sigma^{\delta}\sum_{j}\sum_{x_j\leq n \leq x_{j+1}}2a_n\big(\log x_j\big)e^{-\sigma \log x_j}\\
	&\lesssim 
	\sigma^{\delta}\sum_{j}(\log x_j)^{\delta}e^{-\sigma \log x_j}= 
	\sum_{j}2^{j\delta}e^{-2^j},
\end{align*}
where in the second inequality we have applied condition \eqref{condan}. 
Finally,
$$\sum_{n=1}^{\infty}\sigma^{\delta}\big(\log n\big)a_ne^{-\sigma \log n}\lesssim\sum_{j}2^{j\delta}e^{-2^j}<\infty.$$
\vskip-10pt\end{proof}
We can generalise the ideas from the proof of this lemma to the spaces $\text{Bloch}_{\mu}(\C_+)$. We state without proof the following result.
\begin{lema}
Let $\mu$ be such that $\omega(\sigma)=\sigma^{\delta}$, $\delta>0$. Let $f(s)=\sum_{n\geq1}a_nn^{-s}$ be a Dirichlet series bounded in $\C_1$, with non negative coefficients. Then, $f\in\text{Bloch}_{\mu}(\C_+)$ if and only if, for $x\geq2$, 
\begin{equation}\label{condan1}
\sum_{x\leq n\leq x^2}a_n=O\left(\frac{1}{\log(x) \omega(1/\log x)}\right).
\end{equation}
\end{lema}

In analogy with \cite[Lemma 3]{bommier}, we can detail some of the properties of Dirichlet series belonging to the $\text{Bloch}_{\mu}$-spaces. One of the properties we are interested in is that of the inclusion relationship between these new Bloch spaces and the somehow `classical' version of the Bloch space in the right half-plane, namely, $\text{Bloch}(\C_+)$. The importance of this inclusion relies on the fact that we aim to find a sufficiently large space of symbols, so that when $g$ lies in this space, the operator $T_g$ is bounded in $\mathcal{A}_{\mu}^p$. Hence, it is somehow natural that the definition of the space of symbols depends on the measure $\mu$. To be able to settle some interesting inclusion relations, we shall look at the following two conditions.
\begin{defi}\label{miramientos}
    Let $\mu$ be a probability measure absolutely continuous  on $(0,\infty)$ and such that its density $h$ satisfies the conditions from \eqref{condicionh}. We say that:
    \begin{itemize}
        \item[a)] the measure $\mu$ satisfies the \emph{$H_1$-condition} if there exists $\alpha>-1$ such that for every $0<\delta<1$, 
    \begin{equation*}
        h(\delta t)\lesssim \delta^{\alpha}h(t),\quad \text{a.e. } t\in[0,1];
    \end{equation*}
    \item[b)] the measure $\mu$ satisfies the \emph{$H_2$-condition} if there exists $c>0$ such that
    \[
    h(x)\geq ch(\sigma)
    \]
    for a.e. $x\in[\sigma/2,\sigma]$, $0<\sigma\leq1$.
    \end{itemize}  
   \end{defi}
Observe that the measures $\mu_{\alpha}$ from \eqref{macaden1} satisfy these conditions.

\begin{prop}\label{inclusion1}
    Let $\mu$ be an absolutely continuous probability measure  on $(0,\infty)$ such that $0\in\text{supp}(\mu)$. Then,
     \begin{itemize}
         \item[a)] If $\mu$ satisfies the $H_1$-condition, 
         \[
    \emph{Bloch}(\C_+)\cap\mathcal{D}\subset\emph{Bloch}_{\mu}(\C_{+})\cap\mathcal{D}.  
    \]
    In particular, $\mathcal{H}^{\infty}\subsetneq \emph{Bloch}_{\mu}(\C_{+})\cap\mathcal{D} $. %\textcolor{red}{Moreover, there exists a probability measure $\mu$ on $(0,\infty)$ such that the above inclusion is strict. }
    \item[b)] If $\mu$ satisfies the $H_2$-condition,
    \[
    \emph{Bloch}_{\mu}(\C_{+})\cap\mathcal{D}\subset\emph{Bloch}(\C_+)\cap\mathcal{D}.
    \]
     \end{itemize}
                \end{prop}

\begin{proof}
We begin by proving $a)$. In Remark \ref{sismuertos} it was shown that $\mathcal{H}^{\infty}\subset\text{Bloch}(\C_+)\cap\mathcal{D}$. For the strict inclusion, it suffices to consider a function $f\in\text{Bloch}(\D)\setminus H^{\infty}(\D)$, that is, if $f$ is not  bounded in $\D$ and $\sup_{z\in\D}(1-|z|^2)|f'(z)|<\infty$, and set $g(s)=f(2^{-s})$. The function $g$ is in $\text{Bloch}(\C_+)$ but not in $\mathcal{H}^{\infty}$. 
The function $f(z)=-\log(1-z)$ does the job. 
Now, once the inclusion between the Bloch spaces is settled, the fact that $\mathcal{H}^{\infty}\subsetneq\text{Bloch}_{\mu}(\C_+)$ will follow automatically for $\mu$ a measure satisfying the $H_1$-condition.
    
    \ 
    Let $f\in\text{Bloch}(\C_+)\cap\mathcal{D}$ and let $\mu$ be a measure satisfying the $H_1$-condition. Then,
    \begin{align*}
    \omega(\sigma)^2=\frac{\beta_{\mu}(\sigma)}{h(\sigma)}        &=
    \frac{1}{h(\sigma)}\int_0^{\sigma}(\sigma-t)h(\frac{t}{\sigma}\sigma)dt\\
    &\lesssim \int_0^{\sigma}(\sigma-t)\left(
\frac{t}{\sigma}\right)^{\alpha}dt
\approx
\sigma^2.
    \end{align*}
    % Summarising, we have that $\beta_{\mu}(\sigma)\leq\sigma^2O(h(\sigma))$, $\sigma>0$. 
    Therefore,
    \begin{align*}
        \|f\|_{\text{Bloch}_{\mu}(\C_+)}&=\esssup_{\substack{0<\sigma\leq 1\\ t\in\R}}\omega(\sigma)|f'(\sigma+it)|+\|f\|_{\mathcal{H}^{\infty}(\C_1)}\\&
        \lesssim
        \sup_{\substack{0<\sigma\leq 1\\ t\in\R}}\sigma |f'(\sigma+it)|+\|f\|_{\mathcal{H}^{\infty}(\C_1)}<+\infty.
            \end{align*}
where we have used Remark  
\ref{esmuyutil} and the fact that $f\in\text{Bloch}(\C_+)\cap\mathcal{D}$.

\ 
Regarding $b)$, we find that
\begin{align*}
\beta_{\mu}(\sigma)=\int_{0}^{\sigma}\int_0^th(u)dudt\geq
\int_{\sigma/2}^{\sigma}\int_{\sigma/2}^th(u)dudt&
\geq
c\int_{\sigma/2}^{\sigma}h(\sigma)(t-\sigma/2)dt
\\&\geq
\frac{c}{8}\sigma^2h(\sigma).
\end{align*}
Repeating the argument for part $a)$ we obtain the desired inclusion.
\end{proof}
In the following example we provide an example of an admissible measure $\nu$ defined on $(0,+\infty)$ not satisfying the $H_1$-condition, but such that the inclusion of the $\text{Bloch}(\C_+)$ space in the corresponding $\text{Bloch}_{\nu}(\C_+)$ space is strict (see Example \ref{ejemplofuerte} for a measure satisfying the $H_1$-condition). The next example 
comes to confirm that the forthcoming Theorem \ref{carlesonberg} actually improves the somehow expected sufficient condition $g\in\text{Bloch}(\C_+)$, in order to guarantee the boundedness of $T_g$ on $\mathcal{A}^p_{\mu}$.
\begin{example}\label{ejemplomeasure}
Let $\gamma>1$ and consider the measures $\nu_{\gamma}$ given by $d\nu_{\gamma}(\sigma)=h_{\gamma}(\sigma)d\sigma$, where
	\begin{equation}\label{densitych}
		h_{\gamma}(\sigma)=c_{\gamma}\exp(-\sigma^{-\gamma+1})
		\left(\left(\gamma-1\right)\sigma^{-2\gamma}-\sigma^{-(\gamma+1)}\right),
			\end{equation}
			for $\sigma\in(0,1]$ and $h_{\gamma}(\sigma)=0$ if $\sigma>1$ and $c_{\gamma}$ a suitable normalisation constant so that $h_{\gamma}$ is indeed the density of a probability measure. Then, we have that
   \[
   \emph{Bloch}(\C_+)\cap\mathcal{D}\subsetneq \emph{Bloch}_{\nu_{\gamma}}(\C_+)\cap\mathcal{D}
   \]
\end{example}
The idea to prove this strict inclusion is to consider the second derivative of $\sigma\mapsto\exp(-\sigma^{-\gamma+1})$.
We do not need to know precisely $h_\gamma$ on the remaining part of the right half-line: we only require that $h_\gamma>0$, is continuous and belongs to $L^1(0,+\infty)$. 

Recalling the definition of $\beta_{\mu}$, we see  that, on a neighbourhood of $0$
	\[
	\beta_{\gamma}(\sigma)=\beta_{\nu_{\gamma}}(\sigma)\approx d_{\gamma}\exp(-\sigma^{1-\gamma}).
	\]
	Therefore, by continuity, we have that $\omega(\sigma)\approx \sigma^{\gamma}$ on $(0,1)$. 
	It is plain that for $0<\sigma\leq1$ and $t\in\R$,
	\[
	\sigma^{\gamma}|f'(\sigma+it)|\leq \sigma|f'(\sigma+it)|.
	\]
	Taking the supremum on $t\in\R$ and $\sigma\in(0,1]$ the inclusion follows.	
	To prove the strict inclusion, 	consider the test functions
	\[
	f_{\gamma}(s)=-\sum_{n=2}^{\infty}(\log n)^{\gamma-2}n^{-(s+1)}.
	\]
	Since the $\text{Bloch}_{\mu}$-norm is invariant under vertical translates, we assume $t=0$ and study the behaviour of $f'_{\gamma}$ close to zero. Then,
	\[
	f'_{\gamma}(\sigma)=\sum_{n=2}^{\infty}(\log n)^{\gamma-1}n^{-(\sigma+1)}
 =\zeta^{(\gamma-1)}(\sigma+1)
\approx\frac{1}{\sigma^{\gamma}},
	\]
 where $\zeta^{(\gamma)}$ stands for the $\gamma$-fractional derivative of the Riemann zeta function (see \cite[Lemma 3.1]{olsen2008local}). Consequently, 
	\[
	\sup_{\substack{0<\sigma\leq1\\ t\in\R}}\sigma^{\gamma}|f_{\gamma}'(\sigma+it)|
	\approx
	1.
	\]
 The function $f_{\gamma}$ being bounded on $\C_1$, we have that $f_{\gamma}\in\text{Bloch}_{\nu_{\gamma}}(\C_+)$. However, since $\gamma>1$,
	\[
	\|f_{\gamma}\|_{\text{Bloch}(\C_+)}\geq\sup_{\substack{0<\sigma\leq1\\ t\in\R}}\sigma|f_{\gamma}'(\sigma+it)|\approx\sup_{0<\sigma\leq1}\frac{1}{\sigma^{\gamma-1}}=\infty.
	\]

\begin{remark}
	Observe that we have actually shown that 
	\[
	\text{Bloch}(\C_+)\cap\mathcal{D}\subsetneq\text{Bloch}_{\nu_{\gamma}}(\C+)\cap\mathcal{D}\subsetneq\text{Bloch}_{{\nu_{\beta}}}(\C_+)\cap\mathcal{D},\quad \beta>\gamma>1.
	\]
	\end{remark}
 Next lemma contains some other properties of the $\text{Bloch}_{\mu}$ spaces of Dirichlet series related to the size of the coefficients and the membership of vertical limits.
\begin{lema}\label{propbloch}
    Let $f(s)=\sum_{n\geq1}a_nn^{-s}$ be a Dirichlet series in $\emph{Bloch}_{\mu}(\C_+)$. Then:
    \begin{itemize}
    \item[a)] For every $\chi\in\T^{\infty}$, $f_{\chi}\in\emph{Bloch}_{\mu}(\C_+)$ and $\|f_{\chi}\|_{\emph{Bloch}_{\mu}(\C_+)}= \|f\|_{\emph{Bloch}_{\mu}(\C_+)}$.
      %  \item[\textcolor{red}{c)}] Let $\sigma_0>1/2$. Then, there exists $C=C(\sigma_0)$ such that
        %\[
       % |g'_{\chi}(\sigma+it)|\leq C2^{-\sigma_0}\|g\|_{\emph{Bloch}_{\mu}(\C_+)}, \quad \sigma>\sigma_0, \quad t\in\R,
       % \]
       % for all $\chi\in\T^{\infty}$.
        \item[b)] For every $n\geq3$,
        \[|a_n|\leq {\rm e}\frac{\|f\|_{\emph{Bloch}_{\mu}(\C_+)}}{\log(n)\,\omega(1/\log n)}
        \cdot
        \]
         \item[c)] In particular, $|a_n|=o(n^\epsilon)$ for every $\epsilon>0$.
    %   \item[c)] Let $0<\sigma<1$       \begin{equation*}        |f(\sigma+it)|\leq \left(\int_{\sigma}^1\frac{du}{ \omega(u)}+1\right)\|f\|_{{\emph{Bloch}_{\mu}(\C_+)}}.        \end{equation*}
        
    \end{itemize}
\end{lema}
\begin{proof}
We begin by showing a). Let $\chi$ be a character. Since $f=(f_{\chi})_{\chi^{-1}}$, it is enough to see that 
\begin{equation}\label{loquehay}
    \|f_{\chi}\|_{\text{Bloch}_{\mu}(\C_+)}\leq \|f\|_{\text{Bloch}_{\mu}(\C_+)}.
\end{equation}
 By Kronecker's theorem, there exists a sequence of real numbers $\{\tau_k\}_k$ such that $n^{-i\tau_k}\to\chi(n)$ for all $n\in\N$. Let $f_k(s)=f(s+i\tau_k)$. Using Lemma \ref{born}, for $\varepsilon>0$, we clearly have $$\sup_{k\in\N}\|f_k\|_{\mathcal{H}^{\infty}(\C_{\varepsilon})}<\infty.$$
Therefore, there exists a subsequence, relabelled as $\{f_k\}$, such that $f_k(s)\to f_{\chi}(s)$ locally uniformly on $\C_{\varepsilon}$, for all $\varepsilon>0$. In particular, this implies that $f_k(s)\to f_{\chi}(s)$ and $f_k'(s)\to f_{\chi}'(s)$ for all $s\in\C_+$, as $k\to\infty$. Hence,
\begin{align*}
    \|f_{\chi}\|_{\mathcal{H}^{\infty}(\C_1)}
    \leq
    \sup_{k\in\N}\|f_k\|_{\mathcal{H}^{\infty}(\C_1)}
    =\|f\|_{\mathcal{H}^{\infty}(\C_1)}<\infty
\end{align*}
and, for $\omega\in(0,1]$ fixed,
\begin{align*}
\sup_{t\in\R}\omega(\sigma)|f'_{\chi}(\sigma+it)|
\leq
\sup_{\substack{k\in\N\\ t\in\R}}\omega(\sigma)|f'_{k}(\sigma+it)|
&=
\sup_{\substack{k\in\N\\ t\in\R}}\omega(\sigma)|f'(\sigma+it)|\\
&\leq
\|f\|_{\text{Bloch}_{\mu}(\C_+)}-\|f\|_{\mathcal{H}^{\infty}(\C_1)},
\end{align*}
giving \eqref{loquehay}, as desired.

\ 
For statement $b)$ and $c)$, we use the fact that $f'\in\mathcal{H}^{\infty}(\C_{\varepsilon})$, for $\varepsilon>0$, and apply \cite[Theorem 6.1.1]{queffelecs} so that
\[
|a_nn^{-\varepsilon}\log n|\leq\|T_{\varepsilon}f'\|_{\infty}
\leq
\sup_{t\in\R}|f'(\varepsilon+it)|
\leq 
\frac{\|f\|_{\text{Bloch}_{\mu}(\C_+)}}{\omega(\varepsilon)}\cdot
\]
This is,
\[
|a_n|\leq \frac{\|f\|_{\text{Bloch}_{\mu}(\C_+)}n^{\varepsilon}}{\log(n) \omega(\varepsilon)}\cdot
\]
This implies $c)$ and taking $\varepsilon=\frac{1}{\log n},\  n^{-\varepsilon}={\rm e}^{-1}$, we get $b)$.
\end{proof}
\begin{remark}
	If we consider the density measure $d\mu(\sigma)=h_{\gamma}(\sigma)d\sigma$, $\gamma>1$, where $h_{\gamma}$ is as in \eqref{densitych} from Example \ref{ejemplomeasure}. It turns out that $\omega(\sigma)\approx \sigma^{\gamma}$. Let $f(s)=\sum_{n\geq1}a_nn^{-s}$ belong to $\text{Bloch}_{\mu}(\C_+)$ with norm $1$. In this case, using Lemma \ref{propbloch}, we have that
	$$|a_n|\lesssim (\log n)^{\gamma-1},\quad n\geq2.$$
 In this case, we have a better estimate than $a_n=O(n^{\varepsilon})$, $\epsilon>0$ from Lemma \ref{propbloch}.
\end{remark}

In analogy with the Bloch space of the unit disc, we can also introduce the little oh Bloch$_{\mu}$ version.
\begin{defi}
Let $\theta\in\R$ and consider $f$ an holomorphic function in $\C_{\theta}$. 
We say that $f$ belongs to the little Bloch space $\emph{Bloch}_{\mu,0}(\C_{\theta})$ if
\begin{equation*}
    \lim_{\sigma\to\theta^+}\omega(\sigma-\theta)\sup_{t\in\R}|f'(\sigma+it)|=0.
\end{equation*}
\end{defi}
As we did in the case of the space $\text{Bloch}_{\mu}(\C_+)$, we shall establish some basic properties of this little Bloch space.
\begin{lema}\label{convfu}
    Let $\{f_n\}_n$ be a sequence of functions in $\emph{Bloch}_{\mu,0}(\C_+)$ and consider $f\in\emph{Bloch}_{\mu}(\C_+)$. In addition, assume that
    \begin{itemize}
        \item[i)] $f_n\to f$ uniformly in half-planes $\C_{\varepsilon}$, $\varepsilon>0$.
        \item[ii)] For every $\varepsilon>0$, there exists $\rho=\rho(\varepsilon)$ such that if $\sigma<\rho$, then
        \[
        \omega(\sigma)|f'_n(\sigma+it)|<\varepsilon,
        \]
        for all $t\in\R$ and $n\in\N$.
    \end{itemize}
    Then,
    \[
    f\in\hbox{\rm Bloch}_{\mu,0}(\C_+) \quad\hbox{and}\quad\lim_{n\to\infty}\|f-f_n\|_{\emph{Bloch}_{\mu}(\C_+)}=0.
    \]
\end{lema}
\begin{proof}
Using Cauchy's formula and condition $i)$, we deduce that $f'_n\to f'$ uniformly in half-planes $\C_{\varepsilon}$ and, in particular, it converges pointwise in $\C_{\varepsilon}$. 
This, together with condition $ii)$ implies that $f\in\text{Bloch}_{\mu,0}(\C_+) $. 

On the other hand, using condition $ii)$, there exists $\rho>0$ such that for $\sigma<\rho$,
\[
\omega(\sigma)|f'(\sigma+it)-f'_n(\sigma+it)|<2\varepsilon.
\]
Now, for $\rho\leq\sigma\leq1$, we use that $f'_n\to f'$ uniformly in $\overline{\C_{\rho}}$. %Since $\mu$ satisfies the $H_1$-condition, in the proof of Proposition \ref{inclusion1} it was seen that, in such case, $\omega(\sigma)=O(\sigma)$, $0<\sigma<1$. 
By \eqref{condicionh}, we have that $\omega$ is bounded in $[\rho,1]$, so there exists $n_0=n_0(\varepsilon)$ such that for all $n\geq n_0$, $\sigma\in[\rho,1] $ and $t\in\R$
\[
\omega(\sigma)|f'(\sigma+it)-f'_n(\sigma+it)|<\varepsilon
\]
and the claim follows. Notice also that $\|f-f_n\|_{\mathcal{H}^{\infty}(\C_1)}\to0$ as $n\to\infty$, the conclusion follows.
\end{proof}
The following result is an immediate consequence of the preceding lemma.
\begin{lema}\label{horiz1}
    Let $f\in\emph{Bloch}_{\mu,0}(\C_+)\cap\mathcal{D}$. Then,
    \[
    \lim_{\delta\to0^+}\|f-f_{\delta}\|_{\emph{Bloch}_{\mu}(\C_+)}=0.
    \]
\end{lema}
\begin{proof}
    Consider $\{\delta_n\}$ a sequence of positive real numbers $\delta_n\downarrow 0$. Then, the sequence $f_n(s)=f(s+\delta_n)$ clearly satisfies condition ii) from Lemma \ref{convfu}. By the maximum principle, since the $\delta_n$'s are positive,
    \[
\sup_{t\in\R}\omega(\sigma)|f'_n(\sigma+it)|=
\omega(\sigma)\sup_{t\in\R}|f'(\sigma+\delta_n+it)|\leq\omega(\sigma)\sup_{t\in\R}|f'(\sigma+it)|.
    \]
    Since $f\in\text{Bloch}_{\mu,0}$, we can find $\rho>0$
 small enough so that the supremum on $(0,\rho)$ is controlled by $\varepsilon$ uniformly in $n$. Condition $i)$ follows from the fact that $\sigma_b(f_n)\leq \sigma_b(f)\leq 0$ for all $n$. 
 An application of Lemma \ref{convfu} gives the result.
\end{proof}
\begin{remark}
    The statement of Lemma \ref{horiz1} is, in general, no longer true for functions in the larger space $\text{Bloch}_{\mu}(\C_+)$. 
\end{remark}
\begin{prop}\label{densite}
Let $\mu$ be a measure satisfying the $H_1$-condition, $g$ be a Dirichlet series belonging to the space $\emph{Bloch}_{\mu,0}(\C_+)$ and $\varepsilon>0$. Then, there exists a Dirichlet polynomial $P$ such that
$\|f-P\|_{\emph{Bloch}_{\mu}(\C_+)}<\varepsilon$.
\end{prop}
\begin{proof}
Let $g$ be a Dirichlet series in $\text{Bloch}_{\mu,0}(\C_+)$ and $\varepsilon>0$. 
We recall that $S_Ng$ denotes the $N$-th partial sum of $g$. 
By Lemma \ref{horiz1}, there exists $\delta>0$ such that $\|g-g_{\delta}\|_{\text{Bloch}_{\mu}(\C_+)}<\varepsilon/2$. 
Moreover, since $\mu$ satisfies the $H_1$-condition, we can apply Proposition \ref{inclusion1} so that
\begin{align*}
    \|g-(S_Ng)_{\delta}\|_{\text{Bloch}_{\mu}(\C_+)}
    &\leq
    \|g-g_{\delta}\|_{\text{Bloch}_{\mu}(\C_+)}
    +
    \|g_{\delta}-(S_Ng)_{\delta}\|_{\text{Bloch}_{\mu}(\C_+)}\\
    &\leq \frac{\varepsilon}{2}+C\|g_{\delta}-(S_Ng)_{\delta}\|_{\mathcal{H}^{\infty}}.
\end{align*}
Now, by Lemma \ref{born}, the partial sums of $g$, $S_Ng$, converge to $g$ uniformly in $\overline{\C}_{\delta}$. 
Then, we can find $N=N(g,\delta)$ large enough so that $\|g_{\delta}-(S_Ng)_{\delta}\|_{\mathcal{H}^{\infty}}\leq \varepsilon/2C $ and the conclusion follows.
\end{proof}

\section{Volterra operator on \texorpdfstring{$\text{Bloch}(\C_+)\cap\mathcal{D}$}{}}\label{sec:volterrainbloch}
In this section, we will first characterise the bounded Volterra operators acting on the Bloch space $\text{Bloch}(\C_+)\cap\mathcal{D}$. Then, the description of the compact Volterra operators $T_g$ is settled. 

\ 
We recall that the pointwise evaluation functionals $\delta_s$ are defined, for $s\in\C_+$, as $\delta_sf=f(s)$, where $f\in\text{Bloch}(\C_+)$.
The following estimate of evaluations $\delta_s$, interesting by itself, will be needed in the proof of Theorem \ref{boundednessbloch}.
\begin{prop}\label{functional}
	For $s=\sigma+it, \sigma>0$, let $\delta_s:\emph{Bloch}(\C_+)\cap\mathcal{D}\to\C$ be given by $\delta_s(f)=f(s)$. Then,    
	\begin{equation*}
      \log\left(\frac{1}{\sigma}\right)  \lesssim\|\delta_s\|_{(\emph{Bloch}(\C_+))^*}\lesssim \left(1+  \log\left(\frac{1}{\sigma}\right)\right).
    \end{equation*}
\end{prop}
\begin{proof}
The upper estimate was shown in \cite{queffetal}. We sketch the argument for clarity. Let $f\in\text{Bloch}(\C_+)$ be such that $\|f\|_{\text{Bloch}(\C_+)}=1$. Suppose first that $\sigma\geq1$, then
\[
|f(\sigma+it)|\leq\|f\|_{\mathcal{H}^{\infty}(\C_1)}\leq\|f\|_{\text{Bloch}(\C_+)}.
\]
Suppose now $\sigma<1$, then
\begin{align*}
|f(\sigma+it)|\leq|\int_{\sigma}^1f'(u+it)du|+|f(1+it)|
&\leq
\int_{\sigma}^1\frac{1}{u}du+1\\
&\leq 1+\log\left(
\frac{1}{\sigma}
\right).
\end{align*}
     We prove the lower bound. Let
    \[
    f(s)=\sum_{n=2}^{\infty}\frac{1}{n\log n}n^{-s}.
    \]
    This is, $f$ is a primitive of $\zeta(s+1)-1$. First, we check that $f$ is in $\text{Bloch}(\C_+)$ and its Bloch norm is less or equal than $1$. Indeed,
    \[
    |f'(s)|\leq\sum_{n\geq2}\frac{1}{n^{\sigma+1}}=\zeta(\text{Re}(s)+1)-1.
    \]
    Therefore,
    \[
    |\text{Re}(s)f'(s)|\leq\text{Re}(s)|\zeta(\text{Re}(s)+1)|\leq 2,\quad 0<\text{Re}(s)<1.
    \]    
    Now,
    \begin{align*}
f(\sigma)\geq\int_{\rm e}^{\infty}\frac{x^{-\sigma}}{x\log x}dx
        \geq\int_{\rm e}^{e^{1/\sigma}}\frac{x^{-\sigma}}{x\log x}dx
        \geq
        \frac{1}{\rm e}\log(1/\sigma)\,.
    \end{align*}
\end{proof}
\begin{teor}\label{boundednessbloch}
Let $g\in\mathcal{D}$. The operator $T_g$ is bounded on $\emph{Bloch}(\C_+)$ if and only if $g\in\emph{Bloch}_{\omega}(\C_+)$, where
    $\omega(\sigma)=\sigma(1+\log(1/\sigma))$.
\end{teor}
\begin{proof}
  The boundedness of the $T_g$ occurs if and only if
  \[
  \sup_{\|f\|_{\text{Bloch}(\C_+)}=1}\|T_gf\|_{\text{Bloch}(\C_+)}<\infty.
  \]
  This is, if
  \[
  \sup_{\|f\|_{\text{Bloch}(\C_+)}=1}\sup_{\substack{\sigma>0\\ t\in\R}}\sigma|f(\sigma+it)||g'(\sigma+it)|<\infty.
  \]
  This supremum is finite if and only if 
  \[
  \sup_{\substack{\sigma>0\\ t\in\R\\ s=\sigma+it}}\sigma|g'(\sigma+it)|\|\delta_s\|_{(\rm{Bloch}(\C_+))^*}<\infty.
  \]
  Now, by Proposition \ref{functional}, this supremum being finite is equivalent to the following 
  \begin{equation}\label{Th42proof1}
   \sup_{\substack{\sigma>0\\ t\in\R}}\sigma(1+\log(1/\sigma))|g'(\sigma+it)|<\infty.
  \end{equation}
  Therefore the operator is bounded if and only if $g\in\rm{Bloch}_{\omega}(\C_+)$ for $\omega=\sigma(1+\log(1/\sigma))$, as desired.
  
  Indeed, when condition \eqref{Th42proof1} is satisfied: clearly $g$ is bounded in $\C_1$ and this condition implies that $g\in\rm{Bloch}_{\omega}(\C_+)$. Conversely when $g\in\rm{Bloch}_{\omega}(\C_+)$, \eqref{Th42proof1} is satisfied on the strip $\C_+\setminus\overline{\C_1}$. 
  The function $g$ being bounded, by Lemma \ref{born}, on every $\C_\varepsilon$,  $g'$ is bounded as well on $\C_1$.
Moreover we have for instance, for $\sigma\geq3$, thanks to Lemma \ref{propbloch},
$$\sigma(1+\log(1/\sigma))|g'(\sigma+it)|\lesssim\sum_{n\geq2}\frac{\log(n)\log(\log(n))}{n^2}\,\cdot$$
and \eqref{Th42proof1} is satisfied.
\end{proof}

We can provide a \emph{compactness} version of the preceding theorem.
\begin{teor}
 Let $g\in\mathcal{D}$.
The operator $T_g$ is compact on $\emph{Bloch}(\C_+)$ if and only if $g\in\emph{Bloch}_{\omega,0}(\C_+)$, where
    $\omega=\sigma(1+\log(1/\sigma))$.
\end{teor}

\begin{proof} Thanks to the preceding theorem, and since $\rm{Bloch}_{\omega,0}(\C_+)\subset \rm{Bloch}_{\omega}(\C_+)$, we can assume in the following that $g\in\rm{Bloch}_{\omega}(\C_+)$.

Assume that $T_g$ is compact, its adjoint $T_g^\ast$ is compact as well. 
On the other hand, the functional $f\mapsto\Delta_s(f)= f'(s)$ is clearly bounded on $\rm{Bloch}(\C_+)$ with norm less than or equal to $1/\sigma$.
In other words the functional  ${\rm Re}(s)\Delta_{s}$ lies in the unit ball of the dual of  $\rm{Bloch}(\C_+)$ for every $s\in\C_+$.
Moreover for every polynomial $f$, we have Re$(s)\Delta_s(f)\to0$ when ${\rm Re}(s)\to0^+$. 
Therefore, by a standard density argument, for every sequence $(s_n)$ satisfying $\sigma_n={\rm Re}(s_n)\to0^+$, the sequence $\sigma_n\Delta_{s_n}$ is weak$^\ast$ converging to $0$.
By compactness of $T_g^\ast$, we have $\|T_g^\ast(\sigma_n\Delta_{s_n})\|\to0$ as a functional on $\rm{Bloch}(\C_+)$.
But clearly 
$$T_g^\ast(\sigma_n\Delta_{s_n})=\sigma_n g'(s_n)\delta_{s_n}$$
and we get $\displaystyle \sigma_n |g'(s_n)|.\|\delta_{s_n}\|\to0$ which expresses exactly that $g\in\rm{Bloch}_{\omega,0}(\C_+)$, since $(s_n)$ is arbitrary.

Now assume that $g\in\rm{Bloch}_{\omega,0}(\C_+)\subset \rm{Bloch}(\C_+)$ and take any sequence $(f_n)$ in the unit ball of $\rm{Bloch}(\C_+)$. 
By a now standard normal family argument (see \cite{bayart2}), this sequence converges uniformly on every half plane $\C_a$, with $a>0$, {\sl a fortiori} pointwise, to some $f$ in the unit ball of $\rm{Bloch}(\C_+)$. 
We may assume that $f=0$ in the sequel. 
Therefore, fixing $\varepsilon>0$, the hypothesis on $g$ implies that there exists $\eta>0$ such that  
$${\rm Re}(s)\|\delta_{\rm{Re}(s)}\|_{(\rm{Bloch}(\C_+))^\ast}|g'(s)|\leq\varepsilon$$ as soon as ${\rm Re}(s)\leq\eta$.
We get, when ${\rm Re}(s)\leq\eta$,
$${\rm Re}(s)\big|\big(T_g(f_n)\big)'(s)\big|={\rm Re}(s)|g'(s)||f_n(s)|\leq{\rm Re}(s)|g'(s)|\|\delta_{\rm{Re}(s)}\|_{(\rm{Bloch}(\C_+))^\ast}\leq\varepsilon\,.$$

On the other hand, we know that $f_n\to 0$ uniformly in $\C_\eta$ so
$$\sup_{{\rm Re}(s)>\eta}{\rm Re}(s)\big|\big(T_g(f_n)\big)'(s)\big|\leq\|g\|_{\rm{Bloch}(\C_+)}\,\sup_{{\rm Re}(s)>\eta}|f_n(s)|\longrightarrow0\,.$$
Gathering everything we conclude that $\|T_g(f_n)(s)\|_{\rm{Bloch}(\C_+)}$ converges to $0$.\end{proof}

\begin{remark}
   Notice that the proof of Proposition \ref{functional} provides an example showing that, for $\omega$ as above, we have the strict inclusion $\text{Bloch}_{\omega}\subsetneq\text{Bloch}(\C_+)$. Indeed, it was seen there that
   $f(s)=\sum_{n=2}^{\infty}\frac{1}{n\log n}n^{-s}$ belongs to $\text{Bloch}(\C_+)$. However, when $\sigma\to0^+$, we have
   \[
\sigma\log\left(1+1/\sigma\right)|f'(\sigma)|\sim\log(1+1/\sigma)\longrightarrow+\infty\,.
   \]
\end{remark}

\section{The spaces \texorpdfstring{$\mathcal{H}^p$ and $\mathcal{A}_\mu^p$}{apmu}}\label{sec:apmuetal}
\subsection{The \texorpdfstring{$\mathcal{H}^p$}{}-spaces}
One of the key ingredients which enabled in \cite{HedLinSeip} the introduction of the Hardy spaces of Dirichlet series is an observation due to H. Bohr. 
Essentially, Bohr observed that a Dirichlet series $f(s)=\sum_{n\geq1}a_nn^{-s}$ could be seen as a function on the infinite polydisc $\D^{\infty}$ by means of the fundamental theorem of arithmetic. 
Indeed, given $n\geq2$, it can be (uniquely) written as $n=p_1^{\alpha_1}\cdots p_r^{\alpha_r}$, where $\alpha_j\in\N\cup\{0\}$ for all $j$, and $\{p_j\}$ is the increasing sequence of prime numbers. Then, letting $z_j=p_j^{-s}$, 
\[
f(s)=\sum_{n=1}^{\infty}a_n(p_1^{-s})^{\alpha_1}\cdots (p_r^{-s})^{\alpha_r}=\sum_{n=1}^{\infty}a_nz_1^{\alpha_1}\cdots z_r^{\alpha_r}.
\]
This allows us to write a Dirichlet series $f$ as a function $\mathcal{B}f$ in the infinite polydisc 
\[
\mathcal{B} f(z_1,z_2,\ldots)=\sum_{\substack{n\geq1\\ n=p_1^{\alpha_1}\cdots p_r^{\alpha_r}}}a_nz_1^{\alpha_1}\dots z_r^{\alpha_r}.
\]
Let $m_{\infty}$ be the normalised Haar measure on the infinite polytorus $\T^{\infty}$, that is, the product of the normalised Lebesgue measure of the torus $\T$ in each variable. Fix $p\in[1,\infty)$. We define the space $H^p(\T^{\infty})$ as the closure of the analytic polynomials in the $L^p(\T^{\infty},m_{\infty})$-norm. Now, given a Dirichlet polynomial $P$, its Bohr lift $\mathcal{B}P$ belongs to the space $H^p(\T^{\infty})$. This allows us to define its $\mathcal{H}^p$-norm as
\[
\|P\|_{\mathcal{H}^p}
=
\|\mathcal{B} P\|_{H^p(\T^{\infty})}.
\]
The $\mathcal{H}^p$-spaces of Dirichlet series are defined as the completion of the Dirichlet polynomials in the $\mathcal{H}^p$-norm.

\subsection{The \texorpdfstring{$\mathcal{A}^p_{\mu}$}{}-spaces}
 
In \cite{pascal}, a family of Bergman spaces of Dirichlet series was introduced. Namely, there the authors considered an admissible measure $\mu$ on $(0,\infty)$ and defined the norm for a Dirichlet polynomial $P$ as
\begin{equation*}
    \|P\|_{\mathcal{A}^p_{\mu}}
    =
    \left(
    \int_0^{\infty}\|P_{\sigma}\|_{\mathcal{H}^p}^pd\mu(\sigma)
    \right)^{\frac1p},
\end{equation*}
where $P_{\sigma}(s)=P(s+\sigma)$. Then, the space $\mathcal{A}_{\mu}^p$ is the result of taking the completion of the Dirichlet polynomials $\mathcal{P}$ with respect to this norm. The resulting space is a Banach space of Dirichlet series in the half-plane $\C_{1/2}$ as it was shown in \cite[Theorem 5]{pascal}. For any function in the space, we compute its norm  as (\cite[Theorem 6]{pascal})
\[
\|f\|_{\mathcal{A}_{\mu}^p}=
\lim_{c\to0^+}\|f_c\|_{\mathcal{A}_{\mu}^p}.
\]
We shall pay special attention to the density measures. This is, the measures $d\mu(\sigma)=h(\sigma)d\sigma$, where $h\geq0$, $0\in\text{supp}(h)$ and $\|h\|_{L^1(\R_+)}=1$. Then, we shall denote these spaces as $\mathcal{A}_{\mu_h}^p$. There is a certain family of densities which in the case $p=2$ gives rise to a family of Hilbert spaces of Dirichlet series first studied by McCarthy in \cite{mcarthy}. These are the $\mathcal{A}^2_{\mu_{\alpha}}$ spaces associated to the family of densities
\[
h(\sigma)=\frac{2^{\alpha+1}}{\Gamma(\alpha+1)}\sigma^{\alpha}e^{-2\sigma},\quad \alpha>-1.
\]
In the Hilbert space case, we recover a weighted Bergman space of Dirichlet series and the norm can be easily computed through the coefficients of the series. In order to ensure good properties for this space, we must impose some restrictions on these weights.
\begin{defi}\label{defweig}
    Let $\mu$ be a probability measure on $(0,\infty)$ such that $0\in\emph{supp}(\mu)$. For $n\geq1$, we define the Bergman weight as
    \begin{equation}\label{weight}
        w_n:=\int_0^{\infty}n^{-2\sigma}d\mu(\sigma).
    \end{equation}
\end{defi}
Hence, we can also define the weighted Bergman spaces $\mathcal{A}^2_{\mu}$ as
\begin{equation}\label{apmu}
    \mathcal{A}_{\mu}^2=\left\{
    f(s)\in\mathcal{D}: f(s)=\sum_{n\geq1}a_nn^{-s}, \|f\|=\left(
     \sum_{n\geq1}|a_n|^2w_n
    \right)^{1/2}<\infty
    \right\}.
\end{equation}
In fact, this definition coincides with the one consisting on taking the completion of the Dirichlet polynomials in the $\mathcal{A}^2_{\mu}$-norm from \eqref{apmu}.

\ 

%Note that the restriction $0\in\text{supp}(\mu)$ prevents the case $w_n=1$ for all $n$ from happening and, consequently, we exclude the space $\mathcal{H}^2$ from the definition of the Bergman spaces. \textcolor{red}{Pourquoi pas $\mu=\delta_0$? }

McCarthy (\cite{mcarthy}) proved the following property for the weights from Definition \ref{defweig}:
\[
\forall \varepsilon>0, \exists c>0: \forall n\geq1,\quad w_n>cn^{-\varepsilon}.
\]
We present an extra property of these weights which will be needed later in our discussion and which, in particular, implies the McCarthy condition.
\begin{lema}\label{proppesos}
    Let $w_n$ be as in \eqref{weight}. Then, $w_{mn}\geq w_mw_n$ for all $m,n\in\N$.
\end{lema}
\begin{proof}
    The claim is a consequence of the following fact due to Chebyshev: let $\mu$ be a probability measure on $(0,\infty)$ and $f$, $g$ two real functions on $(0,\infty)$ with the same monotony. Then,
    \[
    \int_0^{\infty}fgd\mu\geq \int_0^{\infty}fd\mu\int_0^{\infty}gd\mu.
    \]
    Let $n,m\in\N$. We shall apply this to the functions $f(\sigma)=n^{-2\sigma}$ and $g(\sigma)=m^{-2\sigma}$, which, clearly have the same monotony. Therefore:
    \[
    w_{mn}=\int_0^{\infty}(mn)^{-2\sigma}d\mu(\sigma)\geq\int_0^{\infty}m^{-2\sigma}d\mu(\sigma)\int_0^{\infty}n^{-2\sigma}d\mu(\sigma)
    =w_mw_n.
    \]

\end{proof}
Notice that these weights satisfy the McCarthy condition from above. Indeed, the Lemma \ref{proppesos} gives us that $w_{2^k}\geq (w_2)^k$. Since $w_2>0$, there exists $\varepsilon>0$ such that $w_2\geq 2^{-\varepsilon}$. Hence,
\[
w_n\geq 2^{-\varepsilon} n^{-\varepsilon}.
\]

\subsection{Auxiliary results: $\mathcal{A}^p_{\mu}$ tools} 
We now present some basic properties of these Bergman spaces which are nothing but the corresponding version of the known results for $\mathcal{H}^p$.
\begin{lema}\label{a1}
    Let $1\leq p<\infty$. Given $f(s)=\sum_{n\geq1}a_nn^{-s}$ belonging to $\mathcal{A}^p_{\mu}$, we have that
    \[
    |a_1|\leq \|f\|_{\mathcal{A}^p_{\mu}}.
    \]
\end{lema}
The next lemma follows easily from both the corresponding $\mathcal{H}^p$ inequality (\cite[Theorem C]{brevigcorr}) and the  definition of the $\mathcal{A}_{\mu}^p$-norm for polynomials. 
The result for general $\mathcal{A}_{\mu}^p$ follows by a standard density argument.
 
\begin{lema}\label{trahor}
    Let $1\leq p<\infty$. Then, for every $f\in\mathcal{A}_{\mu}^p$,
    \[
    \|f_{\sigma}\|_{\mathcal{A}_{\mu}^p}\leq \|f_x\|_{\mathcal{A}_{\mu}^p},\quad \sigma\geq x\geq0.
    \]
\end{lema}
The next lemma also follows immediately from its $\mathcal{H}^p$ version (\cite[Theorem C]{brevigcorr}). It is worth mentioning that for $p>1$ the result can be proven without passing through the $\mathcal{H}^p$ theorem. Indeed, using the fact that $(e_n)_n=n^{-s}$ is a Schauder basis in $\mathcal{A}_{\mu}^p$ (\cite{aleolsen}, \cite{pascal}) and applying Abel's summation formula, we can deduce the result up to some constant.
\begin{lema}\label{prinul}
    Let $1\leq p<\infty$. If $f(s)=\sum_{n\geq N}a_nn^{-s}$ for some $N\geq2$, then
    \[
  \|f_{\sigma}\|_{\mathcal{A}^p_{\mu}}\leq N^{-\sigma} \|f\|_{\mathcal{A}^p_{\mu}}
    \]
    for all $\sigma\geq0$.
\end{lema}
 
 In the following result we show that right horizontal translation strongly improves the integrability.
It can be seen as a version in the framework of the Dirichlet series of the fact that in the classical framework of the unit disc, the blow up operator $f\mapsto f(rz)$ maps the classical spaces of analytic functions (Hardy, Bergman,...) into $H^\infty$. 
Here we cannot expect to always reach $\mathcal{H}^\infty$ nevertheless we have:

\begin{lema}\label{horiz}
    Let $1\leq  p,q<\infty$ and $\varepsilon>0$. Then, $T_{\varepsilon}:f\in\mathcal{A}_{\mu}^p\to f(\cdot+\varepsilon)\in\mathcal{H}^q$ is bounded.
\end{lema}
\begin{remark}\label{remarko}
In fact, we are mostly interested in knowing that the norm on the embedding is uniformly bounded when $\varepsilon$ is far from zero. 

From the proof below, we have
\[
\|T_{\varepsilon}f\|_{\mathcal{H}^p}\lesssim
\frac{\|f\|_{\mathcal{A}_{\mu}^p}}{\mu([0,\varepsilon])^{\frac{1}{p}}}\cdot
\]
Since $0\in\text{supp}(\mu)$, the integral is uniformly bounded when $\varepsilon>\varepsilon_0>0$.
\end{remark}

\begin{proof}
We first recall that (see \cite{pascal}) that, given $\varepsilon>0$, there exists $C_1(\varepsilon)>0$ such that
\begin{equation}\label{embAH1}
\forall f\in\mathcal{A}_{\mu}^1\,,\qquad\|T_{\varepsilon}f\|_{\mathcal{H}^1}\leq C_\varepsilon\|f\|_{\mathcal{A}_{\mu}^1}
\end{equation}
where $C_\varepsilon=1/\mu([0,\varepsilon])$.
Indeed, for sake of completeness, we give a quick argument: take any Dirichlet polynomial $f$. 
By definition, 
$$\|f\|_{\mathcal{A}_{\mu}^1}\geq\int_0^\varepsilon \|f_\sigma\|_{\mathcal{H}^1}\,d\mu\geq\int_0^\varepsilon \|f_\varepsilon\|_{\mathcal{H}^1}\,d\mu=\|f_\varepsilon\|_{\mathcal{H}^1}\mu([0,\varepsilon])$$ 
since $\sigma>0\mapsto\|f_\sigma\|_{\mathcal{H}^1}$ is non-increasing.
We can conclude by the density of the Dirichlet polynomials both in the Hardy and Bergman spaces and we get \eqref{embAH1}.

On another hand, we know that $T_\varepsilon$ is bounded from $\mathcal{H}^1$ to $\mathcal{H}^2$. Indeed the Helson inequality expresses that every $f\in\mathcal{H}^1$ satisfies
$$\sum_n\dfrac{|a_n|^2}{d(n)}\leq\|f\|_{\mathcal{H}^1}^2,$$
where $d(n)$ denotes the number of divisors of $n$, and since $d(n)=O\big(n^{2\varepsilon})$ we get
$$\forall f\in\mathcal{H}^1\,,\qquad\|T_{\varepsilon}f\|_{\mathcal{H}^2}\leq k_\varepsilon\|f\|_{\mathcal{H}^1}$$
where $k_\varepsilon>0$.

By induction, for every integer $q\geq1$, we have 
$$
\forall f\in\mathcal{H}^1\,,\qquad\|T_{ q\varepsilon}f\|_{\mathcal{H}^{2^q}}\leq k^{{2-2^{1-q}}}_\varepsilon\|f\|_{\mathcal{H}^1}.
$$
Let us see this. Clearly, the inequality holds for the case $q=1$. Let us assume that such is the case for the case $q$, then, using that $(T_{\varepsilon}(h))^2=T_{\varepsilon}(h^2)$
\begin{align*}
   \|T_{(q+1)\varepsilon}f\|_{\mathcal{H}^{2^{q+1}}} 
   =
   \|T_{q\varepsilon}(T_{\varepsilon}f)\|_{\mathcal{H}^{2^{q+1}}}
   &=
   \|T_{q\varepsilon}((T_{\varepsilon}f)^2)\|_{\mathcal{H}^{2^{q}}}^{1/2}\\
   &\leq (k_{\varepsilon}^{2-2^{1-q}})^{1/2}\|(T_{\varepsilon}f)^2\|_{\mathcal{H}^1}^{1/2}
  \\& =(k_{\varepsilon}^{2-2^{1-q}})^{1/2}\|T_{\varepsilon}f\|_{\mathcal{H}^2}.
\end{align*}
Now, applying the case $q=1$, we obtain
\begin{align*}
   \|T_{(q+1)\varepsilon}f\|_{\mathcal{H}^{2^{q+1}}}
   \leq 
   (k_{\varepsilon}^{2-2^{1-q}})^{1/2}\|T_{\varepsilon}f\|_{\mathcal{H}^2}
  & \leq 
   k_{\varepsilon}(k_{\varepsilon}^{2-2^{1-q}})^{1/2}
   \|f\|_{\mathcal{H}^1}\\
   &=k_{\varepsilon}^{2-2^{-q}}\|f\|_{\mathcal{H}^1}.
\end{align*}
Hence, with $C_{\varepsilon,q}=k^{{2-2^{1-q}}}_{\frac{\varepsilon}{q}}$,
\begin{equation}\label{embHq}
\forall f\in\mathcal{H}^1\,,\qquad\|T_{\varepsilon}f\|_{\mathcal{H}^{2^q}}\leq C_{\varepsilon,q}\|f\|_{\mathcal{H}^1}.
\end{equation}
Gathering, applying \eqref{embAH1} and \eqref{embHq} with $\varepsilon/2$, we obtain
$$\forall f\in\mathcal{A}_{\mu}^1\,,\qquad\|T_{\varepsilon}f\|_{\mathcal{H}^{2^q}}\leq C_{\varepsilon/2}C_{\varepsilon/2,q}\|f\|_{\mathcal{A}_{\mu}^1}$$
and this obviously implies the conclusion of the lemma.\end{proof}

From this we can also state:
\begin{lema}\label{normApineg} 
Let $1\leq p<\infty$. Then, for every $f\in\mathcal{A}_{\mu}^p$, we have
$$\|f\|_{\mathcal{A}_{\mu}^p}^p=\int_0^{\infty}\|f_{\delta}\|_{\mathcal{H}^p}^pd\mu(\delta).$$
\end{lema}
\begin{proof}
Let $f\in\mathcal{A}^p_{\mu}$. From Lemma \ref{horiz}, for every $c>0$, $f_c\in\mathcal{H}^p$ so, thanks to the $\mathcal{H}^p$ version of Lemma \ref{trahor},
$$\|f_c\|_{\mathcal{A}_{\mu}^p}^p=\int_0^{\infty}\|f_{c+\delta}\|_{\mathcal{H}^p}^pd\mu(\delta)\leq\int_0^{\infty}\|f_{\delta}\|_{\mathcal{H}^p}^pd\mu(\delta).$$
Taking the limit when $c\to0^+$, by \cite[Theorem 6]{pascal}, we get one inequality.

\ 
For the converse one, take $c=1/m$, $m\in\N$. Now, first using once more the $\mathcal{H}^p$ version of Lemma \ref{trahor}, and then applying the Monotone Convergence Theorem to $g_m(\cdot)=\|f_{1/m+\cdot}\|_{\mathcal{H}^p}^p$, we obtain 
\begin{align*}
    \|f\|_{\mathcal{A}^p_{\mu}}^p
    \geq \|f_{\frac1m}\|_{\mathcal{A}^p_{\mu}}^p
    =
\int_0^{\infty}\|f_{\frac1m+\delta}\|_{\mathcal{H}^p}^pd\mu(\delta)\to   \int_0^{\infty}\|f_{\delta}\|_{\mathcal{H}^p}^pd\mu(\delta),\quad m\to\infty. 
\end{align*}
\end{proof}
\subsection{Auxiliary results: $\mathcal{H}^p$ tools}
In order to ease the reading of the proof of the main results in Section \ref{sec:volterraonberg}, we find convenient to devote this subsection to recall all those results related to the spaces $\mathcal{H}^p$ which will be needed there.

\ 

The following lemma, due to Bayart, is a well-known result in the theory of $\mathcal{H}^p$-spaces.
\begin{lema}{\emph{(\cite[Lemma 5]{bayart2})}}\label{hp}
    Let $\mu$ be a finite Borel measure on $\R$. Then,
    \begin{equation*}
        \|f\|_{\mathcal{H}^p}^p \mu(\R)=\int_{\T^{\infty}}|f_{\chi}(it)|^pd\mu(t)dm_{\infty}(\chi).
    \end{equation*}
\end{lema}
The next Theorem \ref{breper} and Theorem \ref{breper1} were established in \cite{brevigcorr} using some classical results from $H^p$-spaces theory such as \cite{feff} and \cite{BGS}.  It is also worth mentioning that the same results were proven in \cite{chen} using different arguments.

\ 
For $\tau\in\R$, we define the cone
\[
\Gamma_{\tau}=\{\sigma+it\in\C_+:|t-\tau|<\sigma\}.
\]
Then, for any analytic function $f$ in $\C_+$, we define its non-tangential maximal function as
\begin{equation}
    f^*(\tau):=\sup_{s\in\Gamma_{\tau}}|f(s)|.
\end{equation}

We can use this maximal function to provide a lower estimate of the $\mathcal{H}^p$-norm.
\begin{teor}\label{breper}\emph{(\cite[Theorem A]{brevigcorr})}
   Fix $0<p<\infty$. There is a constant $C_p\geq1$ such that
   \begin{equation*}
       \int_{\T^{\infty}}\int_{\R}|f^*_{\chi}(\tau)|^p\frac{d\tau}{\pi(1+\tau^2)}dm_{\infty}(\chi)\leq C_p\|f\|_{\mathcal{H}^p}^p
   \end{equation*}
   for every $f\in\mathcal{H}^p$.
\end{teor}
The $\mathcal{H}^p$-norm can also be written in terms of the following square function.
\begin{defi}\label{sqmax}
Let $f$ be a holomorphic function in $\C_+$. Then, its square function is given by
\[
Sf(\tau)=
\left(
\int_{\Gamma_{\tau}}|f'(\sigma+it)|^2d\sigma dt
\right)^{1/2},\quad\tau\in\R.
\]
\end{defi}
\begin{teor}\label{breper1}\emph{(\cite[Theorem B]{brevigcorr})}
 Fix $0<p<\infty$.  There exist constants $A_p>0$ and $B_p>0$ such that
   \begin{equation}\label{squareform}
       A_p\|f\|_{\mathcal{H}^p}^p\leq
       \int_{\T^{\infty}}\int_{\R}(Sf_{\chi}(\tau))^p\frac{d\tau}{\pi(1+\tau^2)}dm_{\infty}(\chi)
       \leq B_p\|f\|_{\mathcal{H}^p}^p
   \end{equation}
   for every $f\in\mathcal{H}^p$ vanishing at infinity.
\end{teor}
The following auxiliary lemma will also be useful in the proof of the sufficient condition.
\begin{lema}{\emph{(\cite[Lemma 5.2]{brevig})}}\label{lematecnraro}
Let $\varphi$ be a function and $\mu$ a positive measure on the vertical strip $\mathbb{S}=\{s:0<\emph{Re}(s)<1\}$. Then,
\begin{equation}\label{lemeq1}
\int_{\R}\int_0^1|\varphi(\sigma+it)|d\mu(\sigma,t)\approx \int_{\R}\int_{\mathbb{S}\cap\Gamma_{\tau}}|\varphi(\sigma+it)|\frac{1+t^2}{\sigma}d\mu(\sigma,t)\frac{d\tau}{1+\tau^2}\cdot
\end{equation}
\end{lema}

\section{Littlewood-Paley-type estimates}\label{section:LP}
When proving the boundedness of the Volterra operator in Banach spaces of Dirichlet series, having an integral expression for the norm involving the derivative turns out to be quite useful. 
This is one of the interests behind seeking Littlewood-Paley identities in this setting. 
The  Littlewood-Paley identity for the $\mathcal{A}_{\mu_h}^p$ was first established in \cite[Theorem 7]{pascal} for the case $p=2$. The proof for the range $p\in[1,\infty)$ was given in \cite[Lemma 4.5]{wang}. 
The proof is based on the Littlewood-Paley identity for $\mathcal{H}^p$ (see %\cite[Proposition 4]{bayart} and 
\cite[Theorem 5.1]{bayart1}). 
Let us recall this identity in $\mathcal{H}^p$:

\begin{prop}\label{LPHp}
    Let $p\geq1$ and let $\eta$ be a probability measure on $\R$. Then, for every $f\in\mathcal{H}^p$ we have that
    \begin{equation*}
        \|f\|_{\mathcal{H}^p}^p
        \approx
|f(+\infty)|^p+\int_{\T^{\infty}}\int_{\R}\int_0^{\infty}|f_{\chi}(\sigma+it)|^{p-2}|f'_{\chi}(\sigma+it)|^2\sigma d\sigma d\eta(t)dm_{\infty}(\chi)
    \end{equation*}
\end{prop}

Now we state a version in the framework of weighted Bergman spaces.

\begin{prop}\label{LP}
Let $p\geq 1$ and let $\eta$ be a probability measure on $\R$. Then, for every $f\in\mathcal{A}^p_{\mu}$, we have that
\begin{equation*}
\|f\|_{\mathcal{A}^p_{\mu}}^p\approx|f(+\infty)|^p+\int_{\T^{\infty}}\int_{\R}\int_0^{\infty}|f_{\chi}(\sigma+it)|^{p-2}|f'_{\chi}(\sigma+it)|^2\beta_{\mu}(\sigma)d\sigma d\eta(t)dm_{\infty}(\chi)
\end{equation*}
where, we recall
\[
\beta_{\mu}(\sigma)=\int_0^{\sigma}\mu([0,t])dt=\int_0^{\sigma}\int_0^td\mu(u)dt.
\]
\end{prop}
%\begin{remark} Notice that when $d\mu(\sigma)=h(\sigma)d\sigma$ we recover exactly the Littlewood-Paley identity from \cite[Lemma 4.5]{wang}.\end{remark}The proof of this proposition is quite similar to one given in \cite{wang}. 

Some particular difficulties appear in the proof of this result and other theorems below.
Let us mention that it is not that clear how to argue by density once we have the inequalities for Dirichlet polynomials or even functions in $\mathcal{H}^p$.
This explains why we often use some translate of the functions and make a crucial use of Lemma \ref{horiz}.

\begin{proof}[Proof of Proposition \ref{LP}]
Let $f\in\mathcal{A}^p_{\mu}$. 
Thanks to Lemma \ref{normApineg}, we can explicitly compute the $\mathcal{A}_{\mu}^p$-norm of $f$ as
\begin{equation}\label{lodelanorma}
    \|f\|_{\mathcal{A}_{\mu}^p}^p=\int_0^{\infty}\|f_{\delta}\|_{\mathcal{H}^p}^pd\mu(\delta).
\end{equation}
We begin by applying Proposition \ref{LPHp} to the $\mathcal{H}^p$ function $f_{\delta}$ so that
\begin{align*}
\|f_{\delta}\|_{\mathcal{H}^p}^p&\approx|f(+\infty)|^p+
\int_{\T^{\infty}}\int_{\R}\int_0^{\infty}|f_{\chi}(\delta+\sigma+it)|^{p-2}|f'_{\chi}(\delta+\sigma+it)|^2\sigma d\sigma d\eta(t)dm_{\infty}(\chi)\\
&
\approx|f(+\infty)|^p+
\int_{\T^{\infty}}\int_{\R}\int_{\delta}^{\infty}
|f_{\chi}(u+it)|^{p-2}|f'_{\chi}(u+it)|^2(u-\delta)du d\eta(t)dm_{\infty}(\chi).
\end{align*}
Now, carrying out the change of variables $\delta+\sigma=u$ in the above integral and then integrating with respect to the measure $d\mu$ gives
\begin{align*}
\int_0^{\infty} &\|f_{\delta}\|_{\mathcal{H}^p}^pd\mu(\delta)
   \\&
   \approx|f(+\infty)|^p\\&+
\int_0^{\infty}\!\!\int_{\T^{\infty}}\!\int_{\R}\int_{\delta}^{\infty}\!\!
|f_{\chi}(u+it)|^{p-2}|f'_{\chi}(u+it)|^2(u-\delta)du d\eta(t)dm_{\infty}(\chi)d\mu(\delta).
\end{align*}
Applying Fubini in the last integral, we find that
\begin{align*}
\int_{\T^{\infty}}&\int_{\R}\int_0^{\infty}\int_{\delta}^{\infty}
|f_{\chi}(u+it)|^{p-2}|f'_{\chi}(u+it)|^2(u-\delta)dud\mu(\delta)d\eta(t)dm_{\infty}(\chi)\\
&=\int_{\T^{\infty}}\int_{\R}\int_0^{\infty}\int_0^u
(u-\delta)d\mu(\delta)|f_{\chi}(u+it)|^{p-2}|f'_{\chi}(u+it)|^2dud\eta(t)dm_{\infty}(\chi)\\
&=
\int_{\T^{\infty}}\int_{\R}\int_0^{\infty}\beta_{\mu}(u)|f_{\chi}(u+it)|^{p-2}|f'_{\chi}(u+it)|^2dud\eta(t)dm_{\infty}(\chi).
\end{align*}
Using this together with \eqref{lodelanorma}
gives the desired result.\end{proof}

\section{Volterra operator on \texorpdfstring{$\mathcal{A}_\mu^p$}{}}\label{sec:volterraonberg}

\subsection{A sufficient condition for boundedness}
In what follows, we shall repeatedly use the following remark:

\begin{equation}\label{opertransl}
(T_gf)_{\sigma}(s)=(T_gf)(s+\sigma)=T_{g_{\sigma}}f_{\sigma}(s),\qquad \sigma\geq0.
\end{equation}
%\smallskip

The goal of this subsection is to prove the following theorem.

\begin{teor}\label{carlesonberg}
   Let $1\leq p<\infty$ and let $\mu$ be an admissible measure satisfying the conditions from Definition \ref{mubloch}. Suppose that $g\in\emph{Bloch}_{\mu}(\C_+)\cap\mathcal{D}$. Then, the operator $T_g$ is bounded on $\mathcal{A}^p_{\mu}$. Moreover, $$\|T_g\|_{\mathcal{L}(\mathcal{A}_{\mu}^p)}\lesssim \|g\|_{\emph{Bloch}_{\mu}(\C_+)},$$ where the underlying constants may depend on $p$, but not on $g$.
\end{teor}
Inspired by the argument for the $\mathcal{H}^p$ setting from \cite{brevig}, in order to establish Theorem \ref{carlesonberg}, we will use a Carleson measure type condition which is actually sufficient for the boundedness of $T_g$ on $\mathcal{A}^p_{\mu}$. 
%The proof of the next theorem is strongly based on the one of \cite[Theorem 5.1]{brevig}. % which, in turn, is inspired in an argument appeared in \cite{pau}. 
We begin by proving the corresponding Carleson condition version in the setting of the spaces $\mathcal{A}_{\mu}^p$. The next lemmas will be useful in the proof of such theorem, so we prove them beforehand.
\begin{lema}\label{tgf}
 Let $p>1$. Consider $f$ be a Dirichlet polynomial and $g\in\mathcal{A}^p_{\mu}$. Then, $T_gf\in\mathcal{A}^p_{\mu}$.
\end{lema}
\begin{proof}
Let $g(s)=\sum_{n\geq1}a_nn^{-s}$ and fix $p>1$. By linearity, it suffices to prove the result for Dirichlet monomials. Hence, for $k\geq1$, take ${\rm e}_k(s)=k^{-s}$. For every $s$ in a suitable right half-plane, we have
$$
T_{g}({\rm e}_k)(s)={\rm e}_{k}(s)\sum_{n\geq1}a_n(g)\dfrac{\log(n)}{\log(k)+\log(n)}{\rm e}_{n}(s).
$$
Since $\big({\rm e}_n\big)_{n\geq1}$ is a Schauder basis of $\mathcal{A}^p_{\mu}$ and $g\in\mathcal{A}^p_{\mu}$, the series $\sum_{n\geq1}a_n{\rm e}_{n}$ converges in $\mathcal{A}^p_{\mu}$. For each $k\in\N$, we consider the sequence $\{\lambda_{n,k}\}_n$, where
\[
\lambda_{n,k}
=
\dfrac{\log(k)}{\log(k)+\log(n)}\cdot
\]
Now,
\begin{align*}
    \sum_{n\geq1}a_n\dfrac{\log(n)}{\log(k)+\log(n)}{\rm e}_{n}(s)
    =
    \sum_{n\geq1}a_ne_n
    -
    \sum_{n\geq1}a_n\lambda_{n,k}e_n
\end{align*}
The first term converges in $\mathcal{A}^p_{\mu}$, by what was said before. Regarding the second one,  apply Abel's summation and then use both the fact that, for fixed $k$, $\{\lambda_{n,k}\}$ is non-increasing in $n$, and the fact that $\{e_n\}$ is a Schauder basis for $\mathcal{A}^p_{\mu}$.
\end{proof}
\begin{lema}\label{loco1}
For all $c>0$, there exists $\alpha_{\mu}(c)>0$ such that, if $g\in\mathcal{A}^2_{\mu,\infty}$ and
   \[
I(g):=\left(\int_{\T^\infty}\int_{\R}\int_0^{\infty}|g_{\chi}'(\sigma+it)|^2\beta_{\mu}(\sigma)d\sigma \frac{dt}{1+t^2}dm_{\infty}(\chi)\right)^{1/2},
   \]
  then
   \begin{equation*}
       \|g_c\|_{\mathcal{H}^2}\leq \alpha_{\mu}(c)I(g).
   \end{equation*}
\end{lema}
\begin{proof}
 Let $g\in\mathcal{A}^2_{\mu,\infty}$ and $c>0$. Using Remark \ref{remarko} for $p=2$, we have that
\begin{equation*}
\|g_{c}\|_{\mathcal{H}^2}^{2}\lesssim\frac{\|g\|_{\mathcal{A}^2_{\mu}}^2}{\mu((0,c])}\cdot
\end{equation*}
Now, applying Proposition \ref{LP} to the $\mathcal{A}^2_{\mu}$ function $g$, as $g(+\infty)=0$, the lemma follows.
\end{proof}
\begin{lema}\label{loco2}
    Let $g\in\mathcal{A}^p_{\mu,\infty}$. Then, 
    \[
    |g'(\sigma+it)|\lesssim 2^{-\sigma}\|g_{1/6}\|_{\mathcal{H}^2},\quad \sigma\geq1.
    \]
\end{lema}
\begin{proof}
Let $\sigma\geq1$. By Cauchy's inequality and the fact that $\mathcal{H}^2\hookrightarrow\mathcal{H}^{\infty}(\C_{1/2+c})$, $c>0$, we have 
    \begin{align*}
        |g'(\sigma+it)|\leq \frac1c\|g\|_{\mathcal{H}^{\infty}(\C_{\sigma-c})}
       & \lesssim
        \|g_{\sigma-c-(1/2+c)}\|_{\mathcal{H}^2}\\
        &=\|(g_c)_{\sigma-1/2-3c}\|_{\mathcal{H}^2}.
    \end{align*}
Now, by the $\mathcal{H}^p$-version of Lemma \ref{prinul}, if $\sigma-1/2-3c>0$, $c<1$,
\[
|g'(\sigma+it)|\lesssim
\|(g_c)_{\sigma-1/2-3c}\|_{\mathcal{H}^2}
\leq 2^{-(\sigma-1/2-3c)}\|g_{c}\|_{\mathcal{H}^2}.
\]
Taking $c=1/6$ gives the desired result.
\end{proof}
\begin{teor}\label{copiaux}
    Let $1\leq p<\infty$. Assume both $g\in\mathcal{A}^p_{\mu}$ and the existence of a positive constant $C=C(g,p)$ such that for every $f\in\mathcal{A}^p_{\mu}$
    \begin{align}\label{carlesontype}
    \int_{\T^\infty}\int_{\R}\int_0^{\infty}|f_{\chi}(\sigma+it)|^p|g_{\chi}'(\sigma+it)|^2\beta_{\mu}(\sigma)d\sigma \frac{dt}{1+t^2}&dm_{\infty}(\chi)
        \\&\leq C(g,p)^2\|f\|_{\mathcal{A}^p_{\mu}}^p\nonumber.
    \end{align}
Then the operator $T_g$ is bounded on $\mathcal{A}^p_{\mu}$.         
    
    Conversely, if the operator $T_g$ acts boundedly on $\mathcal{A}_{\mu}^p$ and $p\geq2$, then $g\in\mathcal{A}^p_{\mu}$ and condition \eqref{carlesontype} is satisfied for certain constant $C(g,p)$. Furthermore, if
\[
C(g,p)^2:= \sup_{\|f\|_{\mathcal{
A}^p_{\mu}}=1}\int_{\T^\infty}\int_{\R}\int_0^{\infty}|f_{\chi}(\sigma+it)|^p|g_{\chi}'(\sigma+it)|^2\beta_{\mu}(\sigma)d\sigma \frac{dt}{1+t^2}dm_{\infty}(\chi),
        \]
        then
        $C(g,p)\lesssim\|T_g\|_{\mathcal{L}(\mathcal{A}^p_{\mu})}$ if $p\geq2$ and $\|T_g\|_{\mathcal{L}(\mathcal{A}^p_{\mu})}\lesssim C(g,p)$ for all $p\geq1$.
\end{teor}
\begin{proof}[Proof of Theorem \ref{copiaux}]
We first assume that \eqref{carlesontype} is satisfied for every $f\in\mathcal{A}^p_{\mu}$. 
In particular, testing $f=1$, we know that $g\in\mathcal{A}_{\mu}^2$ by Proposition \ref{LP}, hence $g\in\mathcal{A}_{\mu}^{\max(p,2)}$. Let $m=\max(p,2)$. 
Since  $g\in\mathcal{A}^m_{\mu}$, for every Dirichlet polynomial $f$, by Lemma \ref{tgf}, $T_gf$ belongs to $\mathcal{A}^m_{\mu}$. Moreover, thanks to  Lemma \ref{horiz}, for every $\delta>0$ and every Dirichlet polynomial $f$,
\begin{equation}\label{claimg}
(T_g(f))_\delta\in\mathcal{H}^p\;.
\end{equation}
By \cite[Proposition 2]{pascal}, for almost every $\chi\in\T^{\infty}$, the measure $\nu_{g,\chi}$, given by
$$d\nu_{g,\chi}(\sigma,t)=|g'_{\chi}(\sigma+it)|^2\beta_{\mu}(\sigma)d\sigma\frac{dt}{1+t^2},$$
is well defined in the right half-plane $\C_+$. Observe that 
\[
d\nu_{g_{\varepsilon},\chi}(\sigma,t)\not=d\nu_{g,\chi}(\varepsilon+\sigma,t),\quad\varepsilon>0.
\]
Let us also consider the measure $\mu_{g,\chi}$, given by
$$d\mu_{g,\chi}(\sigma,t)=|g'_{\chi}(\sigma+it)|^2\sigma d\sigma\frac{dt}{1+t^2},$$
which is also well defined in the right half-plane $\C_+$ for almost every $\chi\in\T^{\infty}$ (see \cite[Theorem 6]{bayart2}).

\smallskip

The case $p=2$ of the Theorem follows easily from Proposition \ref{LP} and that $(T_gf)'=fg'$.

\
Now, we treat the case $p>2$. Take a Dirichlet polynomial $f$. Thanks to \eqref{claimg}, we can apply the Littlewood-Paley identity (\textrm{i.e.} Proposition \ref{LP}) to $(T_g(f))_\delta$ for any $\delta>0$, which gives, together with the observation in \eqref{opertransl}:
$$\|(T_gf)_\delta\|_{\mathcal{A}_{\mu}^p}^p\approx
   \int_{\T^{\infty}}\int_{\R}\int_0^{\infty}\!|(T_gf)_{\delta,\chi}(\sigma+it)|^{p-2}|f_{\delta,\chi}(\sigma+it)|^2d\nu_{g_\delta,\chi}(\sigma,t)dm_{\infty}(\chi).$$

Now, from Hölder's inequality for the conjugate exponents $q=\frac{p}{p-2}$ and $q'=\frac p2$, we get
\begin{align}\label{holderLP}%\tag{HLP}
    \|(T_gf)_\delta\|_{\mathcal{A}_{\mu}^p}^p&\lesssim
   \left(
\int_{\T^{\infty}}\int_{\R}\int_0^{\infty}|(T_gf)_{\delta,\chi}(\sigma+it)|^{p}d\nu_{g_\delta,\chi}(\sigma,t)dm_{\infty}(\chi)
   \right)^{\frac{p-2}{p}}\\
   &\times
   \left(\int_{\T^{\infty}}\int_{\R}\int_0^{\infty}
|f_{\delta,\chi}(\sigma+it)|^pd\nu_{g_\delta,\chi}(\sigma,t)dm_{\infty}(\chi)
   \right)^{\frac2p}\notag.
\end{align}
Now, point out that when $\psi\in\mathcal{A}_{\mu}^p$, since $\beta_h$ is non-decreasing, for every $\delta>0$:
\begin{align*}
\int_{\T^{\infty}}&\int_{\R}\int_0^{\infty}
\!|\psi_{\delta,\chi}(\sigma+it)|^p d\nu_{g_\delta,\chi}(\sigma,t)dm_{\infty}(\chi)\\
&\leq\int_{\T^{\infty}}\!\int_{\R}\int_0^{\infty}
\!|\psi_{\chi}(\delta+\sigma+it)|^p|g'_{\chi}(\delta+\sigma+it)|^2\beta_h(\delta+\sigma)d\sigma\frac{dt}{1+t^2}dm_{\infty}(\chi)\\
&=\int_{\T^{\infty}}\int_{\R}\int_\delta^{\infty}\,
|\psi_{\chi}(u+it)|^p|g'_{\chi}(u+it)|^2\beta_h(u)du\frac{dt}{1+t^2}dm_{\infty}(\chi).
\end{align*}
Having this, an application of condition \eqref{carlesontype} gives
\begin{align}\label{observation}
\int_{\T^{\infty}}\int_{\R}\int_0^{\infty}
|\psi_{\delta,\chi}(\sigma&+it)|^p d\nu_{g_\delta,\chi}(\sigma,t)dm_{\infty}(\chi)\leq C(g,p)^2 \|\psi\|_{\mathcal{A}_{\mu}^p}^p.
\end{align}
Applying inequality \eqref{observation} to each integral in \eqref{holderLP}, we get
$$
\|(T_gf)_\delta\|_{\mathcal{A}_{\mu}^p}^p\lesssim C(g,p)^2\|T_gf\|_{\mathcal{A}_{\mu}^p}^{p-2}\|f\|_{\mathcal{A}_{\mu}^p}^{2}.
$$
Taking the limit when $\delta\to0^+$, and, since $T_g(f)\in\mathcal{A}^p_{\mu}$, rearranging, we obtain 
$$\|T_gf\|_{\mathcal{A}_{\mu}^p}\lesssim C(g,p)\|f\|_{\mathcal{A}_{\mu}^p}.
$$
Then, for this range of $p$, the result follows and $\|T_g\|_{\mathcal{L}(\mathcal{A}^p_{\mu})}\lesssim C(g,p)$.
\medskip

Suppose now that $1\leq p<2$. It suffices again to prove the theorem for Dirichlet polynomials $f$ and extend the result by a density argument. 

Recall $T_g(f)\in\mathcal{A}^p_\mu$ and we know from Lemma \ref{normApineg} that
    \begin{equation}\label{lobo}
          \|T_gf\|_{\mathcal{A}^p_{\mu}}^p
   =
\int_0^{\infty}\|(T_g(f))_{\delta}\|_{\mathcal{H}^p}^pd\mu(\delta).  
    \end{equation}
    We begin by using Theorem \ref{breper1} and the definition of $f^*$ so that, setting $s=\sigma+it$
\begin{align*}
\|(T_gf)_{\delta}\|_{\mathcal{H}^p}^p
&\approx\int_{\T^{\infty}}\int_{\R}\left(\int_{\Gamma_{\tau}}|f_{\chi}(\delta+s)|^2|g_{\chi}'(\delta+s)|^2d\sigma dt
    \right)^{\frac p2}\frac{d\tau}{1+\tau^2}dm_{\infty}(\chi)\\
&=\int_{\T^{\infty}}\int_{\R}\left(\int_{\Gamma_{\tau}}
    |f_{\delta,\chi}(s)|^{2-p}|f_{\delta,\chi}(s)|^{p}|g_{\delta,\chi}'(s)|^2d\sigma dt
    \right)^{\frac p2}\frac{d\tau}{1+\tau^2}dm_{\infty}(\chi)\\
    &\leq 
    \int_{\T^{\infty}}\int_{\R}((f_{\delta})_{\chi}^*(\tau))^{\frac{(2-p)p}{2}}\left(\int_{\Gamma_{\tau}}
   |f_{\delta,\chi}(s)|^{p}|g_{\delta,\chi}'(s)|^2d\sigma dt
    \right)^{\frac p2}\frac{d\tau}{1+\tau^2}dm_{\infty}(\chi).
\end{align*}
Now, we make use of Hölder's inequality with conjugate exponents $q=\frac2p$ and $q'=\frac{2}{2-p}$. Then, applying Theorem \ref{breper} to the first integral below, we obtain the following:
\vspace{-0.1cm}
\begin{align}\label{quess}
\|\notag(T_gf)_{\delta}\|_{\mathcal{H}^p}^p\notag&\lesssim\left(\int_{\T^{\infty}}\int_{\R}|(f_{\delta})_{\chi}^*(\tau)|^p\frac{d\tau}{1+\tau^2}dm_{\infty}(\chi)
      \right)^{\frac{2-p}{2}}
      \\\notag
      &\times \left(\int_{\T^{\infty}}\int_{\R}\int_{\Gamma_{\tau}}\!
       |f_{\delta,\chi}(s)|^{p}|g_{\delta,\chi}'(s)|^2d\sigma dt\frac{d\tau}{1+\tau^2}dm_{\infty}(\chi)
      \right)^{\frac p2}\\
      &\lesssim \|f_{\delta}\|_{\mathcal{H}^p}^{\frac{(2-p)p}{2}}\!
\left(\int_{\T^{\infty}}\int_{\R}\int_{\Gamma_{\tau}}\!
       |f_{\delta,\chi}(s)|^{p}|g_{\delta,\chi}'(s)|^2d\sigma dt\frac{d\tau}{1+\tau^2}dm_{\infty}(\chi)
      \right)^{\frac p2}
    .
 \end{align}
Let $\mathbb{S}=\{s\in\C:0<\text{Re}(s)<1\}$. We shall split the estimate of the triple integral in two parts: one on $\Gamma_{\tau}\cap\mathbb{S}$ and another one on $\Gamma_{\tau}\setminus\mathbb S$. Hence, for the first term (the one on $\Gamma_{\tau}\cap\mathbb{S}$), applying Lemma \ref{lematecnraro}, we obtain
\begin{align*}
\int_{\T^{\infty}}\int_{\R}\int_{\Gamma_{\tau}\cap\,\mathbb{S}}  & |f_{\delta,\chi}(\sigma+it)|^{p}|g_{\delta,\chi}'(\sigma+it)|^2d\sigma dt\frac{d\tau}{1+\tau^2}dm_{\infty}(\chi)\\& 
\lesssim \int_{\T^{\infty}}\int_{\R}\int_0^{\infty}
|f_{\delta,\chi}(\sigma+it)|^{p}|g'_{\delta,\chi}(\sigma+it)|^2\sigma d\sigma\frac{dt}{1+t^2}dm_{\infty}(\chi).
\end{align*}
Carrying out the change of variables $\delta+\sigma=u$ in the latter integral we find that
\begin{align}\label{escandalo1}
\int_{\T^{\infty}}&\int_{\R}\int_{\Gamma_{\tau}\cap\,\mathbb{S}}   |f_{\delta,\chi}(\sigma+it)|^{p}|g_{\delta,\chi}'(\sigma+it)|^2d\sigma dt\frac{d\tau}{1+\tau^2}dm_{\infty}(\chi))
      \\\notag
&\lesssim\int_{\T^{\infty}}\int_{\R}\int_{\delta}^{\infty}(u-\delta)
      |f_{\chi}(u+it)|^{p}|g_{\chi}'(u+it)|^2du\frac{dt}{1+t^2}dm_{\infty}(\chi)=:M_1.
\end{align}
Regarding the term on $\Gamma_{\tau}\setminus\mathbb S$, 
\begin{align*}
\int_{\T^{\infty}} \int_{\R}&\int_{\Gamma_{\tau}\setminus\mathbb S}|f_{\delta,\chi}(\sigma+it)|^{p}|g_{\delta,\chi}'(\sigma+it)|^2d\sigma dt\frac{d\tau}{1+\tau^2}dm_{\infty}(\chi)\\
&\leq\|\delta_1\|_{(\mathcal{A}^p_{\mu})^*}\|f\|_{\mathcal{A}^p_{\mu}}^p
\int_{\T^{\infty}}\int_{\R}\int_{\Gamma_{\tau}\setminus\mathbb S} |g_{\delta,\chi}'(\sigma+it)|^2d\sigma dt\frac{d\tau}{1+\tau^2}dm_{\infty}(\chi).
\end{align*}
Now, we make use of, first, Lemma \ref{loco2} and then Lemma \ref{loco1} (without loss of generality we can suppose $g(+\infty)=0$, since $g'=(g-g(+\infty))')$, to get
\begin{align*}
\int_{\T^{\infty}}\int_{\R}\int_{\Gamma_{\tau}\setminus\mathbb S} &|g_{\delta,\chi}'(\sigma+it)|^2d\sigma dt\frac{d\tau}{1+\tau^2}dm_{\infty}(\chi)\\&
 \lesssim
 \|g_{1/6}\|_{\mathcal{H}^2}^2
\int_{\T^{\infty}}\int_{\R}\int_{\Gamma_{\tau}\setminus\mathbb S}2^{-\sigma}dtd\sigma \frac{d\tau}{1+\tau^2}dm_{\infty}(\chi)
\\&
=\|g_{1/6}\|_{\mathcal{H}^2}^2
\int_1^{\infty}\pi\sigma2^{-\sigma}d\sigma.
\\
&\lesssim I(g)^2\leq C(g,p)^2
.
\end{align*}
Summing up, for the integral on $\Gamma_{\tau}\setminus\mathbb S$, we have shown that
\begin{align}\label{escandalo2}
    \int_{\T^{\infty}} \int_{\R}&\int_{\Gamma_{\tau}\setminus\mathbb S}|f_{\delta,\chi}(\sigma+it)|^{p}|g_{\delta,\chi}'(\sigma+it)|^2d\sigma dt\frac{d\tau}{1+\tau^2}dm_{\infty}(\chi)\\\notag
&\lesssim
\|f\|_{\mathcal{A}^p_{\mu}}^pC(g,p)^2.
\end{align}
Putting together the estimates \eqref{escandalo1} and \eqref{escandalo2} in \eqref{quess}, we have proven that
\begin{align*}
     \|(T_gf)_{\delta}\|_{\mathcal{H}^p}^p&\lesssim
     \|f_{\delta}\|_{\mathcal{H}^p}^{\frac{(2-p)p}{2}}( M_1+\|f\|_{\mathcal{A}^p_{\mu}}^pC(g,p)^2 )^{p/2}\\&
      \leq
       \|f_{\delta}\|_{\mathcal{H}^p}^{\frac{(2-p)p}{2}}M_1^{p/2}+\|f\|_{\mathcal{A}^p_{\mu}}^{p}C(g,p)^p.
\end{align*}
Then, we integrate on both sides with respect to the measure $\mu$ and using \eqref{lobo}, we obtain
\begin{align*}
    \|&T_gf\|_{\mathcal{A}^p_{\mu}}^p\\&
    \leq 
    \int_0^{\infty}
    \|f_{\delta}\|_{\mathcal{H}^p}^{\frac{(2-p)p}{2}}\left(\int_{\T^{\infty}}\int_{\R}\int_0^{\infty}
     \! |f_{\chi}(\delta+\sigma+it)|^{p}d\mu_{g_{\delta},\chi}(\sigma,t)dm_{\infty}(\chi)
      \right)^{\frac p2}d\mu(\delta)\\
      &+\|f\|_{\mathcal{A}^p_{\mu}}^{p}C(g,p)^p.
\end{align*}
Using Hölder's inequality with respect to the measure $\mu$ with the same conjugate exponents, $q=\frac2p$ and $q'=\frac{2}{2-p}$, we find that
  \begin{align}\label{cuervo}
&\notag\|T_gf\|_{\mathcal{A}^p_{\mu}}^p\lesssim \|f\|_{\mathcal{A}_{\mu}^p}^{\frac{(2-p)p}{2}}
    \\\notag&\times  \left(\int_0^{\infty}\!\!
\int_{\T^{\infty}}\int_{\R}\int_{\delta}^{\infty}\!(u-\delta)
      |f_{\chi}(u+it)|^{p}|g_{\chi}'(u+it)|^2du\frac{d\tau}{1+\tau^2}dm_{\infty}(\chi)
   d\mu(\delta)  \right)^{\frac{p}{2}}\\
   &+\|f\|_{\mathcal{A}^p_{\mu}}^{p}C(g,p)^p
   \!.
\end{align}
Applying Fubini in the integral and condition \eqref{carlesontype}, it follows that
\begin{align*}
\int_{\T^{\infty}}&\int_{\R}\int_0^{\infty}\int_{0}^{u}(u-\delta)d\mu(\delta)
      |f_{\chi}(u+it)|^{p}|g_{\chi}'(u+it)|^2du\frac{d\tau}{1+\tau^2}dm_{\infty}(\chi)\\
    &= \int_{\T^{\infty}}\int_{\R}\int_0^{\infty}\beta_{\mu}(u)
      |f_{\chi}(u+it)|^{p}|g_{\chi}'(u+it)|^2du\frac{d\tau}{1+\tau^2}dm_{\infty}(\chi)\\
      &
      \lesssim C(g,p)^2 \|f\|_{\mathcal{A}_{\mu}^p}^p
      .
   \end{align*}
Plugging this estimate in \eqref{cuervo}:
\begin{align*}
    \|T_gf\|_{\mathcal{A}^p_{\mu}}^p&\lesssim C(g,p)^p\|f\|_{\mathcal{A}_{\mu}^p}^{\frac{(2-p)p}{2}}\|f\|_{\mathcal{A}_{\mu}^p}^{\frac{p^2}{2}}+\|f\|_{\mathcal{A}^p_{\mu}}^{p}C(g,p)^p\\&=C(g,p)^p\|f\|_{\mathcal{A}_{\mu}^p}^p+\|f\|_{\mathcal{A}^p_{\mu}}^{p}C(g,p)^p,
    \end{align*}
which gives the sufficiency and that $\|T_g\|_{\mathcal{L}(\mathcal{A}^p_{\mu})^*}\lesssim C(g,p)$ for the remaining range of $p$.

\
Suppose now that $T_g$ acts boundedly on $\mathcal{A}_{\mu}^p$ for $p>2$. Notice that the case $p=2$ follows from applying Proposition \ref{LP} to $T_gf$ and then using the boundedness of $T_g$ on $\mathcal{A}^2_{\mu}$. Therefore, we assume that $T_g$ is bounded on $\mathcal{A}_{\mu}^p$ for $p>2$. For $\delta>0$, we shall split the proof into two estimates: one on
\[
I_1:=
\int_{\T^{\infty}}\int_{\R}\int_0^1|f_{\chi}(\delta+\sigma+it)|^pd\mu_{g_{\delta},\chi}(\sigma,t)dm_{\infty}(\chi);
\]
and the other one on
\[
I_2
:=
\int_{\T^{\infty}}\int_{\R}\int_1^{\infty}|f_{\chi}(\delta+\sigma+it)|^pd\mu_{g_{\delta},\chi}(\sigma,t)dm_{\infty}(\chi).
\]
For the first one, we begin by using \eqref{lemeq1} from Lemma \ref{lematecnraro}. Define again $\mathbb{S}=\{s\in\C:0<\text{Re}(s)<1\}$. Observe also that $(T_gf)'=fg'$. Hence, for $f\in\mathcal{A}^p_{\mu}$
\begin{align*}
I_1&=\int_{\T^{\infty}}\int_{\R}\int_0^1|f_{\chi}(\delta+\sigma+it)|^pd\mu_{g_{\delta},\chi}(\sigma,t)dm_{\infty}(\chi)\\
    &
\approx\int_{\T^{\infty}}\int_{\R}\int_{\Gamma_{\tau}\cap\,\mathbb{S}}|f_{\chi}(\delta+\sigma+it)|^p|g'_{\chi}(\delta+\sigma+it)|^2d\sigma dt\frac{d\tau}{1+\tau^2}dm_{\infty}(\chi)\\
&\leq\int_{\T^{\infty}}\int_{\R}(f^*_{\delta,\chi}(\tau))^{p-2}\int_{\Gamma_{\tau}}|(T_gf)'_{\chi}(\delta+\sigma+it)|^2d\sigma dt\frac{d\tau}{1+\tau^2}dm_{\infty}(\chi),
\end{align*}
where in the last inequality we have used the definition of the non-tangential maximal function $f^*$ (see Definition \ref{sqmax}).
We obtain the first estimate in a similar fashion as we did for the converse implication. Indeed, we first apply Hölder's inequality for $q=\frac{p}{p-2}$ and $q'=\frac p2$ and we conclude using Theorem \ref{breper} for the first integral below and Theorem \ref{breper1} for the second one below:
\begin{align*}
\int_{\T^{\infty}}\int_{\R}&(f^*_{\delta,\chi}(\tau))^{p-2}\left(\int_{\Gamma_{\tau}}|(T_gf)'_{\chi}(\delta+\sigma+it)|^2d\sigma dt\right)\frac{d\tau}{1+\tau^2}dm_{\infty}(\chi)
   \\
    &\leq \left(\int_{\T^{\infty}}\int_{\R}|f_{\delta,\chi}^*(\tau)|^p\frac{d\tau}{1+\tau^2}dm_{\infty}(\chi)\right)^{\frac{p-2}{p}}\\
    &\times
    \left(
    \int_{\T^{\infty}}\int_{\R}\left(\int_{\Gamma_{\tau}}
    |(T_gf)'_{\chi}(\delta+\sigma+it)|^2d\sigma dt
    \right)^{\frac{p}{2}}
    \frac{d\tau}{1+\tau^2}dm_{\infty}(\chi)\right)^{\frac2p}\\
    &\lesssim\|f_{\delta}\|_{\mathcal{H}^p}^{p-2}\|(T_gf)_{\delta}\|_{\mathcal{H}^p}^2.
 \end{align*}
 Summing up,
 \begin{equation}\label{I1}
     I_1\lesssim \|f_{\delta}\|_{\mathcal{H}^p}^{p-2}\|(T_gf)_{\delta}\|_{\mathcal{H}^p}^2.
 \end{equation}
We now move on to the integral $I_2$. Using that $\sigma\leq \sigma+\delta$, $\delta>0$, and making the change of variables $u=\sigma+\delta$, 
 \begin{align*}
I_2&=\int_{\T^{\infty}}\int_{\R}\int_1^{\infty}|f_{\chi}(\delta+\sigma+it)|^p|g_{\chi}'(\delta+\sigma+it)|^2\sigma d\sigma\frac{dt}{1+t^2}dm_{\infty}(\chi)\\
&\leq\int_{\T^{\infty}}\int_{\R}\int_{1+\delta}^{\infty}|f_{\chi}(u+it)|^p|g_{\chi}'(u+it)|^2u du\frac{dt}{1+t^2}dm_{\infty}(\chi)
\\&\leq\int_{\T^{\infty}}\int_{\R}\int_{1}^{\infty}|f_{\chi}(u+it)|^p|g_{\chi}'(u+it)|^2u du\frac{dt}{1+t^2}dm_{\infty}(\chi).
\end{align*}
By hypothesis, $f\in\mathcal{A}_{\mu}^p$, $p>2$. Then, by the boundedness of the evaluation functional in $\mathcal{A}^p_{\mu}$, we find that
 \[
 |f(\sigma+it)|^p\leq C\|f\|_{\mathcal{A}^p_{\mu}}^{p},\quad \sigma\geq1,
 \]
 with $C>0$ and absolute constant. Let $a_1=g(+\infty)$ be the first coefficient of the Dirichlet series. By the boundedness of $T_g$ on $\mathcal{A}_{\mu}^p$, we have that $g-a_1$ is in $\mathcal{A}^p_{\mu}$ too. In this case, we use Lemma \ref{loco2} and the inclusion between the $\mathcal{A}^p_{\mu}$-spaces ($p>2$), so that 
 \begin{equation*}\label{DeltaH2}
|g'(\sigma+it)|^2\lesssim2^{-\sigma}\|g_{1/6}-a_1\|_{\mathcal{H}^2}^2\lesssim2^{-\sigma}\|g-a_1\|_{\mathcal{A}^2_{\mu}}^2
\leq2^{-\sigma}\|g-a_1\|_{\mathcal{A}^p_{\mu}}^2,\quad \sigma\geq1.
 \end{equation*}
 Then, these observations yield
 \begin{align*}
I_2&\leq\int_{\T^{\infty}}\int_{\R}\int_{1}^{\infty}|f_{\chi}(u+it)|^p|g_{\chi}'(u+it)|^2u du\frac{dt}{1+t^2}dm_{\infty}(\chi)\\ 
&\lesssim
\|f\|_{\mathcal{A}_{\mu}^p}^p\|g-a_1\|_{\mathcal{A}^p_{\mu}}^2
\int_{\T^{\infty}}\int_{\R}\int_1^{\infty}u2^{-u} du\frac{dt}{1+t^2}dm_{\infty}(\chi)\\
&\lesssim\|f\|_{\mathcal{A}_{\mu}^p}^p\|g-a_1\|_{\mathcal{A}^p_{\mu}}^2.
 \end{align*}
Hence, using both \eqref{I1} and the latter estimate on $I_2$, we have that
 \begin{align*}
  I_1+I_2&=    \int_{\T^{\infty}}\int_{\R}\int_0^{\infty}|f_{\chi}(\delta+\sigma+it)|^p|g'_{\chi}(\delta+\sigma+it)|^2\sigma d\sigma\frac{dt}{1+t^2}dm_{\infty}(\chi)\\
      &\lesssim
      \|f_{\delta}\|_{\mathcal{H}^p}^{p-2}\|(T_gf)_{\delta}\|_{\mathcal{H}^p}^2
      +
\|f\|_{\mathcal{A}_{\mu}^p}^p\|g-a_1\|_{\mathcal{A}^p_{\mu}}^2.
 \end{align*}
We shall conclude in an identical way as in the sufficiency. We carry out the change of variables $u=\delta+\sigma$ in the integral $I_1+I_2$. 
 Then, we integrate with respect to the probability measure $\mu$ on both sides. In the left-hand side we will repeat the  argument from the sufficiency and in the right hand integral we use Hölder's inequality with $\alpha=\frac{p}{p-2}$ and $\alpha'=\frac{p}{2}$ for the first term, so that
 \begin{align*}
   \int_{\T^{\infty}}\int_{\R}\int_0^{\infty}\int_{0}^{u}(u-\delta)&d\mu(\delta)
      |f_{\chi}(u+it)|^{p}|g_{\chi}'(u+it)|^2du\frac{dt}{1+t^2}dm_{\infty}(\chi)\\
      &\lesssim \int_0^{\infty}\|f_{\delta}\|_{\mathcal{H}^p}^{p-2}\|(T_gf)_{\delta}\|_{\mathcal{H}^p}^2d\mu(\delta)
      +
\|f\|_{\mathcal{A}_{\mu}^p}^p\|g-a_1\|_{\mathcal{A}_{\mu}^p}^2\\
&\lesssim \|f\|_{\mathcal{A}_{\mu}^p}^{p-2}\|T_gf\|_{\mathcal{A}_{\mu}^p}^{2}
 +
\|f\|_{\mathcal{A}_{\mu}^p}^p\|T_g\|_{\mathcal{L}({\mathcal{A}_{\mu}^p})}^2
 \end{align*}
where in the second term we have used that $\|g-g(+\infty)\|_{\mathcal{A}_{\mu,\infty}^p}\leq\|T_g\|_{\mathcal{L}({\mathcal{A}_{\mu}^p})}$, since $T_g1=g-g(+\infty)$ and the operator $T_g$ is bounded on $\mathcal{A}_{\mu}^p$. Eventually, using the boundedness of $T_g$ on $\mathcal{A}^p_{\mu}$, we find that
 \begin{align*}
    \int_{\T^\infty}\int_{\R}\int_0^{\infty}|f_{\chi}(\sigma+it)|^p|g_{\chi}'(\sigma+it)|^2&\beta_h(\sigma)d\sigma \frac{dt}{1+t^2}dm_{\infty}(\chi)\\&\lesssim
 \|f\|_{\mathcal{A}_{\mu}^p}^p\|T_g\|_{\mathcal{L}({\mathcal{A}_{\mu}^p})}^2.
  \end{align*}
 Taking the supremum over all $f\in\mathcal{A}^p_{\mu}$ with norm $1$ we obtain
 \[
 C(g,p)\lesssim \|T_g\|_{\mathcal{L}({\mathcal{A}_{\mu}^p})}.
 \]
 \vskip-0.5cm
  \end{proof}

\begin{proof}[Proof of Theorem \ref{carlesonberg}]
 The idea is to argue somehow as in the case of the unit disc. 
 Throughout the proof,  we may assume that we work with functions $f$  in $\mathcal{A}_{\mu}^p$ vanishing at infinity, i.e. their first Dirichlet coefficient is $0$. 
 Indeed, we can do this since, given $g$ in $\text{Bloch}_{\mu}(\C_+)$, we claim that 
 $$g\in\bigcap_{1\leq p<\infty}\mathcal{A}_{\mu}^p.$$
 Notice that  for any function $g\in\text{Bloch}_{\mu}(\C_+)$,  $$2\|g\|_{\text{Bloch}_{\mu}(\C_+)}\geq\|g-g(+\infty)\|_{\text{Bloch}_{\mu}(\C_+)}.$$
 Since $\sigma_b(g)\leq0$ (see Lemma \ref{born}), $g_{\delta}\in\mathcal{H}^{\infty}$ for every $\delta>0$ so, in particular, %$g_{\delta}\in\mathcal{A}^p_{\mu}$. 
 $g_{\delta}\in\mathcal{H}^p$ and we can apply the Littlewood-Paley identity from Lemma \ref{LP} to $g_{\delta}$:
 \begin{align*}
     \|g_{\delta}\|_{\mathcal{A}_{\mu}^p}^p
    &\approx
\int_{\T^{\infty}}\!\!\int_{\R}\!\int_0^{\infty}\!\!\beta_{\mu}(\sigma)|g_{\chi}(\delta+\sigma+it)|^{p-2}|g'_{\chi}(\delta+\sigma+it)|^2d\sigma \frac{dt}{1+t^2}dm_{\infty}(\chi)\\
&=
\int_{\T^{\infty}}\!\!\int_{\R}\!\int_0^{1}\!\!\beta_{\mu}(\sigma)|g_{\chi}(\delta+\sigma+it)|^{p-2}|g'_{\chi}(\delta+\sigma+it)|^2d\sigma \frac{dt}{1+t^2}dm_{\infty}(\chi)
\\
&+
\int_{\T^{\infty}}\!\!\int_{\R}\!\int_1^{\infty}\!\!\beta_{\mu}(\sigma)|g_{\chi}(\delta+\sigma+it)|^{p-2}|g'_{\chi}(\delta+\sigma+it)|^2d\sigma \frac{dt}{1+t^2}dm_{\infty}(\chi)\\
&=:B_1+B_2.
\end{align*}
To estimate $B_1$, using both the definition of $\text{Bloch}_{\mu}(\C_+)$ and Lemma \ref{hp} gives,  for $p\geq3$,
\begin{align*}
    B_1
    &\leq 
\|g_{\delta}\|_{\text{Bloch}_{\mu}(\C_+)}^2\!\int_0^{1}\int_{\T^{\infty}}\int_{\R}\!\!|g_{\chi}(\delta+\sigma+it)|^{p-2} \frac{dt}{1+t^2}dm_{\infty}(\chi)h(\sigma)d\sigma\\&
\leq 
\|g\|_{\text{Bloch}_{\mu}(\C_+)}^2\int_0^{1}\|g_{\sigma+\delta}\|_{\mathcal{H}^{p-2}}^{p-2}h(\sigma)d\sigma.
    \end{align*}
Now, using that $h$ is the Radon-Nikodym derivative of the measure $\mu$ and the inclusion relations between the $\mathcal{A}^p_{\mu}$ spaces we obtain, for $p\geq3$,
\begin{align*}
    \|g_{\delta}\|_{\mathcal{A}_{\mu}^p}^p
    &\lesssim
\|g\|_{\text{Bloch}_{\mu}(\C_+)}^2
\int_0^{1}\|g_{\sigma+\delta}\|_{\mathcal{H}^{p-2}}^{p-2}h(\sigma)d\sigma\\
&\leq 
\|g\|_{\text{Bloch}_{\mu}(\C_+)}^2
\int_0^{1}\|g_{\sigma+\delta}\|_{\mathcal{H}^{p-2}}^{p-2}d\mu(\sigma)\\
&\leq \|g\|_{\text{Bloch}_{\mu}(\C_+)}^2\|g_{\delta}\|_{\mathcal{A}_{\mu}^{p-2}}^{p-2}\leq\|g\|_{\text{Bloch}_{\mu}(\C_+)}^2\|g_{\delta}\|_{\mathcal{A}_{\mu}^{p}}^{p-2}.
\end{align*}
Regarding $B_2$, since $g\in\text{Bloch}_{\mu}(\C_+)$, we have $g\in\mathcal{H}^{\infty}(\C_1)$. Recalling that $\beta_{\mu}(\sigma)\leq\sigma$, $\sigma>0$, we obtain, 
\begin{align*}
    B_2&=
\int_{\T^{\infty}}\!\!\int_{\R}\!\int_1^{\infty}\!\!\beta_{\mu}(\sigma)|g_{\chi}(\delta+\sigma+it)|^{p-2}|g'_{\chi}(\delta+\sigma+it)|^2d\sigma \frac{dt}{1+t^2}dm_{\infty}(\chi)\\&
\leq
\|g_{\delta}\|_{\mathcal{H}^{\infty}(\C_1)}^{p-2}
\int_{\T^{\infty}}\!\!\int_{\R}\!\int_1^{\infty}\sigma|g'_{\chi}(\delta+\sigma+it)|^2d\sigma \frac{dt}{1+t^2}dm_{\infty}(\chi)\\
&\leq 
\|g_{\delta}\|_{\mathcal{A}^p_{\mu}}^{p-2}
\int_{\T^{\infty}}\int_{\R}\int_1^{\infty}\sigma|g'_{\chi}(\delta+\sigma+it)|^2d\sigma \frac{dt}{1+t^2}dm_{\infty}(\chi).
\end{align*}
Now, since $g_{\delta}\in\mathcal{H}^p$, by Lemma \ref{prinul}, Cauchy inequality, and Lemma \ref{born}, we have, for $\sigma\geq1$,
\begin{align*}
    |g'_{\chi}(\delta+\sigma+it)|\lesssim 2^{-\sigma+1}\|g'_{\delta}\|_{\mathcal{H}^{\infty}(\C_1)}
    &\leq 2^{-\sigma+1}2\|g_{\delta}\|_{\mathcal{H}^{\infty}(\C_{1/2})}
    \\
    &\leq 4\cdot2^{-\sigma}\|g\|_{\mathcal{H}^{\infty}(\C_{1/2})}
    \\
    &\lesssim 4\cdot2^{-\sigma}\|g\|_{\text{Bloch}_{\mu}(\C_+)}.
    \end{align*}
Plugging this in the estimate for $B_2$
\begin{align*}
    B_2&\lesssim 
    \|g\|_{\text{Bloch}_{\mu}(\C_+)}^2\|g_{\delta}\|_{\mathcal{A}^p_{\mu}}^{p-2}
\int_{\T^{\infty}}\!\!\int_{\R}\!\int_1^{\infty}\sigma2^{-\sigma+1}d\sigma \frac{dt}{1+t^2}dm_{\infty}(\chi)\\
&\lesssim\|g\|_{\text{Bloch}_{\mu}^2(\C_+)}\|g_{\delta}\|_{\mathcal{A}^p_{\mu}}^{p-2}.
\end{align*}
Then, putting together the estimates on $B_1$ and $B_2$, since we know that $\|g_{\delta}\|_{\mathcal{A}_{\mu}^p}<\infty$, we have that
$$\|g_{\delta}\|_{\mathcal{A}_{\mu}^p}\lesssim\|g\|_{\text{Bloch}_{\mu}(\C_+)}, \quad \delta>0.$$
Letting $\delta\to0^+$, the embedding inequality follows for $p\geq3$ and \emph{a fortiori} for every $p\geq1$. 
This, together with Lemma \ref{a1} gives the claim.

\

Now, thanks to Theorem \ref{copiaux}, it suffices to prove \eqref{carlesontype}. Therefore, we estimate the following quantity, where $f\in\mathcal{H}^p$ and $\lim_{\text{Re}s\to+\infty}f(s)=0$:
    \begin{align*}
\int_{\T^{\infty}}\int_{\R}\int_0^{\infty}&\beta_{\mu}(\sigma)|f_{\chi}(\sigma+it)|^p|g'_{\chi}(\sigma+it)|^2d\sigma\frac{dt}{1+t^2}dm_{\infty}(\chi)\\
&
=\int_{\T^{\infty}}\int_{\R}\int_0^{1}\beta_{\mu}(\sigma)|f_{\chi}(\sigma+it)|^p|g'_{\chi}(\sigma+it)|^2d\sigma\frac{dt}{1+t^2}dm_{\infty}(\chi)\\
&
+\int_{\T^{\infty}}\int_{\R}\int_{1}^{\infty}\beta_{\mu}(\sigma)|f_{\chi}(\sigma+it)|^p|g'_{\chi}(\sigma+it)|^2d\sigma\frac{dt}{1+t^2}dm_{\infty}(\chi)\\
&
=: I_1+I_2.
\end{align*}
We begin by estimating $I_2$. Using the definition of the Bloch space $\text{Bloch}_{\mu}(\C_+)$ and Lemma \ref{propbloch} $a)$, we obtain the following:
\begin{align*}
\int_{\T^{\infty}}&\int_{\R}\int_{1}^{\infty}\beta_{\mu}(\sigma)|f_{\chi}(\sigma+it)|^p|g'_{\chi}(\sigma+it)|^2d\sigma\frac{dt}{1+t^2}dm_{\infty}(\chi)\\
&
\leq
\int_{\T^{\infty}}\int_{\R}\int_{1}^{\infty}\beta_{\mu}(\sigma)|f_{\chi}(\sigma+it)|^p\sup_{\substack{\sigma=1\\ t\in\R}}|g'_{\chi}(\sigma+it)|^2
d\sigma\frac{dt}{1+t^2}dm_{\infty}(\chi)\\
&\leq
\frac{1}{\omega(1)^2}\|g\|_{\text{Bloch}_{\mu}(\C_+)}^2
\int_{\T^{\infty}}\int_{\R}\int_{1}^{\infty}\beta_{\mu}(\sigma)|f_{\chi}(\sigma+it)|^p
d\sigma\frac{dt}{1+t^2}dm_{\infty}(\chi)\\
&=\frac{1}{\omega(1)^2}\|g\|_{\text{Bloch}_{\mu}(\C_+)}^2\int_1^{\infty}\|f_{\sigma}\|_{\mathcal{H}^p}^p\beta_{\mu}(\sigma)d\sigma,
\end{align*}
where the last identity follows from Fubini and Lemma \ref{hp}. Now, applying again Fubini as well as the $\mathcal{H}^p$ version of  Lemma \ref{prinul} to $f_{\sigma}=(f_u)_{\sigma-u}$.
\begin{align*}
\int_1^{\infty}\|f_{\sigma}\|_{\mathcal{H}^p}^p\beta_{\mu}(\sigma)d\sigma
&=
\int_1^{\infty}\|f_{\sigma}\|_{\mathcal{H}^p}^p\int_0^{\sigma}(\sigma-u)d\mu(u)d\sigma\\
&
=
\int_0^{\infty}\int_{\sigma>\max(u,1)}\|f_{\sigma}\|_{\mathcal{H}^p}^p(\sigma-u)d\sigma d\mu(u)\\
&\leq 
\int_0^{\infty}\|f_u\|_{\mathcal{H}^p}^p\left(\int_{\sigma>\max(u,1)}2^{-(\sigma-u)}(\sigma-u)d\sigma \right)d\mu(u)\\
&\lesssim 
\int_0^{\infty}\|f_u\|_{\mathcal{H}^p}^pd\mu(u),
\end{align*}
because
\[
I(u)=\int_{\sigma>\max(u,1)}2^{-(\sigma-u)}(\sigma-u)d\sigma 
\] is uniformly bounded with respect to $u$. Since $f\in\mathcal{H}^p$, Lemma \ref{normApineg}, we obtain
\[
I_2\lesssim\int_0^{\infty}\|f_u\|_{\mathcal{H}^p}^ph(u)du\lesssim\int_0^{\infty}\|f_u\|_{\mathcal{H}^p}^pd\mu(u)=\|f\|_{\mathcal{A}^p_{\mu}}^p,
\]
where in the second inequality we have used again that $h$ is the Radon-Nikodym derivative of the measure $\mu$. 

\ 
Now, regarding $I_1$, by both the definition of $\text{Bloch}_{\mu}(\C_+)$ and Lemma \ref{hp}, we have that
\begin{align*}
\int_{\T^{\infty}}\int_{\R}\int_0^{1}&\beta_{\mu}(\sigma)|f_{\chi}(\sigma+it)|^p|g'_{\chi}(\sigma+it)|^2d\sigma\frac{dt}{1+t^2}dm_{\infty}(\chi)\\
&
\leq
\|g\|_{\text{Bloch}_{\mu}(\C_+)}^2
\int_{\T^{\infty}}\int_{\R}\int_0^{1}|f_{\chi}(\sigma+it)|^2h(\sigma)d\sigma\frac{dt}{1+t^2}dm_{\infty}(\chi)\\
&
=
\|g\|_{\text{Bloch}_{\mu}(\C_+)}^2\int_0^1\|f_{\sigma}\|_{\mathcal{H}^p}^ph(\sigma)d\sigma\\
&\leq \|g\|_{\text{Bloch}_{\mu}(\C_+)}^2\|f\|_{\mathcal{A}^p_{\mu}}^p
\end{align*}
where in the last inequality we are again using that $h$ is the Radon-Nikodym derivative of $\mu$. 

Hence, we get inequality \eqref{carlesontype} when $f\in\mathcal{H}^p$. When $f\in\mathcal{A}^p_\mu$, thanks to Lemma \ref{horiz}, we have the inequality \eqref{carlesontype} satisfied by $f_\delta$ for every $\delta>0$. Here we can apply Fatou's lemma and get the inequality for $f$ itself. Applying Theorem \ref{copiaux}, the conclusion follows for every $f\in\mathcal{A}^p_{\mu}$. Moreover, we have that
\[
C(g,p)\lesssim\|g\|_{\text{Bloch}_{\mu}(\C_+)},
\]
which, again by Theorem \ref{copiaux}, gives 
\[
\|T_g\|_{\mathcal{L}(\mathcal{A}_{\mu}^p)}\lesssim \|g\|_{\text{Bloch}_{\mu}(\C_+)}.
\]
\vskip-0.5cm
\end{proof}

Notice that Example \ref{ejemplomeasure} shows that Theorem \ref{carlesonberg} actually improves the somehow expected sufficient condition $g\in\text{Bloch}(\C_+)$.

\ 
The dependence of the sufficient condition on the measure $\mu$ is somehow expected. Our second example comes to confirm this fact.
\begin{example}\label{ejemplofuerte}
There exist both:
\begin{itemize}
\item[i)] an admissible measure $\mu$ so that $\emph{Bloch}_{\mu}(\C_+)\subsetneq\emph{Bloch}(\C_+)$, and
\item[ii)] a Dirichlet series $g$ in $\emph{Bloch}(\C_+)\setminus\emph{Bloch}_{\mu}(\C_+)$ such that $T_g$ is not bounded in $\mathcal{A}^2_{\mu}$.
\end{itemize}
\end{example}
The idea in order to have the first condition satisfied is to choose a measure satisfying the $H_2$-condition. Regarding the second condition, we will consider a Dirichlet series $g$ supported on a lacunary sequence of primes so that condition \eqref{condan} from Lemma \ref{belongcoef} holds. The non-membership of $g$ to the space $\text{Bloch}_{\mu}(\C_+)$ will be ensured by the fact that $T_g1=g$ will fail to belong to $\mathcal{A}^2_{\mu}$. This will give the desired example since $g$ will neither belong to $\text{Bloch}_{\mu}(\C_+)$ since, otherwise, by Theorem \ref{carlesonberg}, $T_g$ would be bounded on $\mathcal{A}^2_{\mu}$.

\ 
Consider the probability measure
\[
d\mu(\sigma)=h(\sigma)\chi_{(0,1]}(\sigma)=\frac{1}{\sigma\log^2\left(\frac{{\rm e}}{\sigma}\right)}\chi_{(0,1]}(\sigma).
\]
This measure satisfies the $H_2$-condition. Indeed, for $x\in[\sigma/2,\sigma]$, $\sigma\in(0,1]$, we have that
\[
h(x)\geq \frac{1}{\sigma\log^2\left(\frac{2{\rm e}}{\sigma}\right)}\gtrsim\frac{1}{\sigma\log^2\left(\frac{{\rm e}}{\sigma}\right)}\cdot
\]
Then, by Proposition \ref{inclusion1}, $\text{Bloch}_{\mu}(\C_+)\subset\text{Bloch}(\C_+)$. It remains to choose the Dirichlet series $g$ conveniently in order to have $g\not\in\mathcal{A}^2_{\mu}$. To this purpose, set $x_j=2^{2^j}$. Clearly, $x_j^2=x_{j+1}$. We let $I_j=[x_{j},x_{j+1}]$. Let $p_j=\min I_j\cap\mathbb{P}$, with $\mathbb{P}$ the collection of prime numbers, and let $\mathcal{P}_0$ be the collection of all such primes. We take as sequence of coefficients the one given by $a_n=\chi_{\mathcal{P}_0}(n)$. We thus obtain an infinite sequence constantly equal to $1$, which clearly fails to be in $\ell^2$. More precisely, let
\[
g(s)=\sum_{p\in\mathcal{P}_0}p^{-s}.
\]
Nonetheless, as
\begin{align*}
    \sum_{x\leq n\leq x^2}a_n\leq \sum_{x_j\leq n\leq x_{j+2}}a_n\leq 2,
\end{align*}
we have that $g\in\text{Bloch}(\C_+)$. Now,
\begin{align*}
\|g\|_{\mathcal{A}^2_{\mu}}^2=\int_0^{\infty}\|g_{\sigma}\|_{\mathcal{H}^2}^2d\mu(\sigma)
\gtrsim
\int_0^{+\infty}\sum_{j=1}^{\infty}2^{-2^{j+1}\sigma}d\mu(\sigma)
&\approx
    \int_0^{\infty}\sum_{n\geq1}\frac{2^{-n\sigma}}{n}d\mu(\sigma)\\
    &
    =-\int_0^{\infty}\log(1-2^{-\sigma})d\mu(\sigma)\\
    &
\approx\int_0^1\frac{d\sigma}{\sigma\log\left( \frac{{\rm e}}{\sigma}\right)}=\infty.
\end{align*}
Hence, $T_g$ is not bounded and, thus, $g\not\in\text{Bloch}_{\mu}(\C_+)$, as desired.
\begin{remark}
Clearly, if a symbol $g$ defines a bounded Volterra operator $T_g$ on $\mathcal{H}^p$, then, by the definition of the $\mathcal{A}^p_{\mu}$-norm, $g$ will define a bounded Volterra operator on $\mathcal{A}^p_{\mu}$ too. Because of this, the sufficient condition from Theorem \ref{carlesonberg} provides an improvement of the sufficient condition $g\in\text{BMOA}(\C_+)$ from \cite{brevig}. Indeed, the symbol $g=\sum_{p\in\mathcal{P}_0}p^{-s}$ from Example \ref{ejemplofuerte} is in  $\text{Bloch}(\C_+)$. However, by \cite[Theorem 5.5]{brevig}, the operator $T_g$ is not bounded in $\mathcal{H}^2$. Therefore, $g$ is not in $\text{BMOA}(\C_+)$ by Theorem \cite[Theorem 2.3]{brevig}.

\end{remark}

\subsection{The necessity matter}
In this section, we will restrict ourselves to the spaces $\mathcal{A}^p_{\alpha}$. This is, $\mu=\mu_{\alpha}$, with $\mu_{\alpha}$ the density measures from \eqref{macaden1}. We shall provide a necessary condition for the boundedness of the Volterra operator $T_g$ acting on the spaces $\mathcal{A}^p_{\alpha}$ (see Theorem \ref{necessity}). The proof of the result is based on a lower estimate of the norm of the evaluation functionals  $\delta_s$, $s\in\C_{1/2}$, on $\mathcal{A}^p_{\alpha}$ done in \cite[Theorem 5.2]{GLQ2} and the upper estimate of the evaluations of the derivatives $\Delta_s$ from Proposition \ref{funder}, which has its own interest.

\ 
Let $f$ be a function in $\mathcal{H}^p$, $1\leq p<\infty$. Consider $s\in\C_{1/2}$. We denote by $\Delta_s$ to the evaluation functional of $f'$ at the point $s$. More precisely, $\Delta_{s}(f)=f'(s)$, $f\in\mathcal{H}^p$.

   Using that the monomials form a orthonormal basis in $\mathcal{H}^2$, it is easy to see that 
   \[
   \|\Delta_s\|_{(\mathcal{H}^2)^*}=(\zeta''(2\text{Re}(s)))^{\frac12},
   \]
which, when $\text{Re}(s)\to\frac12^+$ (see \cite[Theorem 12.22]{apostol}), gives
\[
 \|\Delta_s\|_{(\mathcal{H}^2)^*}\lesssim(2\sigma-1)^{-3/2}.
\]
This can be done thanks to the definition of the $\mathcal{H}^2$-norm, which allows to compute the norm in terms of the coefficients of the series. In the next result, we provide an upper estimate of the norm of $\Delta_s$ in $(\mathcal{H}^p)^*$, $s\in\C_{1/2}$, for the range $1\leq p<\infty$. 
\begin{prop}\label{funder}
Let $1\leq p<\infty$ and $s=\sigma+it$ in $\C_{1/2}$. Then, 
\begin{equation*}
    \|\Delta_s\|_{(\mathcal{H}^p)^*}\lesssim \left(\frac{1}{2\sigma-1}\right)^{\frac{p+1}{p}},\quad \sigma\in(\frac12,2].
\end{equation*}
\end{prop}
\begin{proof}
Let $s=\sigma+it$ belong to $\C_{1/2}$ and such that $\sigma<2$. Using Cauchy's inequality, we have that
    \begin{align*}
        |f'(s)|\leq \frac{1}{\frac12(\sigma-1/2)}\sup_{z\in\partial\D(s,\frac12(\sigma-1/2))}|f(z)|
    &\leq 
    \frac{4}{(2\sigma-1)}\sup_{\text{Re}(z)=\frac12(\sigma+1/2)}|f(z)|.
   \end{align*}
   Applying now that $|f(z)|\leq\|f\|_{\mathcal{H}^p}(\zeta(2\text{Re}(z)))^{\frac1p}$ from \cite[Theorem 3]{bayart2}, we have that
   \begin{align*}
       |f'(s)|\leq \frac{4}{(2\sigma-1)}\|f\|_{\mathcal{H}^p}(\zeta(\sigma+1/2))^{\frac1p}\lesssim 
       \left(\frac{8}{2\sigma-1}\right)^{\frac1p+1}\|f\|_{\mathcal{H}^p}.
   \end{align*}
   Taking the supremum over all 1-norm $\mathcal{H}^p$ functions, we obtain the desired result.
\end{proof}

\begin{teor}\label{necessity}
 Let $1\leq p<\infty$ and $\alpha>-1$. Suppose that $T_g:\mathcal{A}^p_{\alpha}\to \mathcal{A}^p_{\alpha}$ is bounded. Then, $g\in\emph{Bloch}(\C_{1/2})$.
\end{teor}
\begin{proof}
By Remark \ref{esmuyutil}, it suffices to prove that $g\in\mathcal{H}^{\infty}(\C_{3/2})$ and that
\begin{equation}\label{lopri}
    \sup_{\substack{\frac12<\sigma\leq\frac32\\ t\in\R}}(\sigma-1/2)|g'(\sigma+it)|<\infty.
\end{equation}
Since $T_g$ is bounded on $\mathcal{A}^p_{\alpha}$, testing $f=1$, we have that $g\in\mathcal{A}^p_{\alpha}$. In particular, by \cite[Theorem 5]{pascal}, we conclude that $g\in\mathcal{H}^{\infty}(\C_{3/2})$. Hence, it remains to show that \eqref{lopri} holds. To this purpose, let $f$ be a Dirichlet polynomial and $s\in\C_{1/2}$. We begin by using the evaluation functional for the derivative $\Delta_u$ in $\mathcal{H}^p$ at the point $u=s-\sigma\in\C_{1/2}$, where $\text{Re}(s)<3/2$, so
\begin{align*}
    \int_0^{\infty}\|(T_gf)_{\sigma}\|_{\mathcal{H}^p}^pd\mu_{\alpha}(\sigma)
     &\geq|f(s)|^p|g'(s)|^p\int_{0}^{\text{Re}(s)-1/2}\frac{1}{\|\Delta_{s-\sigma}\|_{(\mathcal{H}^p)^*}^p}d\mu_{\alpha}(\sigma).
   %\\&\geq\mu(I_{\delta_0})\frac{|f(s)|^2|g'(s)|^2}{\|\Delta_s\|_{(\mathcal{H}^2)^*}^2}.
\end{align*}
We now estimate from above the left-most integral. Since $T_g$ is bounded on $\mathcal{A}^p_{\alpha}$, we have that $T_g(f)$ belongs to $\mathcal{A}^p_{\alpha}$ and, consequently, $(T_gf)_{\varepsilon}\in\mathcal{H}^p$ for every $\varepsilon>0$. Hence, using both Fatou's Lemma and the boundedness of the operator in $\mathcal{A}^p_{\alpha}$, we have
\begin{align*}
\int_0^{\infty}\!\!\|(T_gf)_{\sigma}\|_{\mathcal{H}^p}^pd\mu_{\alpha}(\sigma)
\leq 
\liminf_{c\to0^+}\!
  \int_0^{\infty}\!\!\|(T_gf)_{\sigma+c}\|_{\mathcal{H}^p}^pd\mu_{\alpha}(\sigma)
    &=\!\liminf_{c\to0^+} \|(T_gf)_c\|_{\mathcal{A}^p_{\alpha}}^p\\
&\leq\|T_gf\|^p_{\mathcal{A}^p_{\alpha}}\\&
\lesssim\|f\|_{\mathcal{A}^p_{\alpha}}^p.
\end{align*}
Putting all together and taking the supremum over the $\mathcal{A}^p_{\mu}$ functions $f$ with norm $1$, we are led to
\begin{equation}\label{estge}
\|\delta_s\|_{(\mathcal{A}^p_{\mu})^*}^p|g'(s)|^p
\int_{0}^{\text{Re}(s)-1/2}\frac{1}{\|\Delta_{s-\sigma}\|_{(\mathcal{H}^p)^*}^p}d\mu_{\alpha}(\sigma)
\leq C,    
\end{equation}
where $C$ depends on $g$ but not on $s$. Then, using Proposition \ref{funder}, we obtain the following
\begin{align}\label{feo}
   \int_{0}^{\text{Re}(s)-1/2}\frac{1}{\|\Delta_{s-\sigma}\|_{(\mathcal{H}^p)^*}^p}d\mu_{\alpha}(\sigma)
   &\approx
   \int_{0}^{\text{Re}(s)-1/2}\frac{1}{\|\Delta_{s-\sigma}\|_{(\mathcal{H}^p)^*}^p}\sigma^{\alpha}d\sigma\nonumber\\
   &\geq 
   \int_{0}^{\text{Re}(s)-1/2}(\text{Re}s-\sigma-1/2)^{p+1}\sigma^{\alpha}d\sigma\nonumber\\
   &=
   (\text{Re}(s)-1/2)^{p+\alpha+2} \int_{0}^{1}(1-t)^{p+1}t^{\alpha}dt,
\end{align}
where in the equality we carried out the change of variables $\sigma=(\text{Re}(s)-1/2)t$, for $t\in(0,1)$. Now, by \cite[Theorem 5.2]{GLQ2}, we have that
\begin{equation*}\label{evaluacion}
    \|\delta_s\|_{(\mathcal{A}^p_{\alpha})^*}^p
\gtrsim
\frac{1}{(2\text{Re}(s)-1)^{2+\alpha}},\quad \alpha>-1.
\end{equation*}
Using this together with \eqref{feo} in \eqref{estge}, we obtain
\[
|g'(s)|^p(2\text{Re}(s)-1)^{p}\leq C.
\]
Therefore,
\[
\sup_{\substack{1/2<\text{Re}(s)\leq3/2 \\t\in\R}}|g'(\sigma+it)|(\text{Re}(s)-\frac12)<\infty,
\]
as desired.
\end{proof}

\subsection{Non-membership in Schatten classes}
Following the ideas from \cite[Theorem 7.2]{brevig}, we adapt the argument to the setting of the $\mathcal{A}_{\mu}^2$ spaces to prove the following:
\begin{teor}
Let $\mu$ be an admissible measure and let $g(s)=\sum_{n\geq1}a_nn^{-s}$ be a non-constant Dirichlet series. Then, $T_g:\mathcal{A}_{\mu}^2\to\mathcal{A}_{\mu}^2$ is not in the Schatten class $S_p$ for any $0<p<\infty$.
\end{teor}
\begin{proof}
    By hypothesis, $g$ is not constant, hence, $n_0=\inf\{n\geq2:a_n\not=0\}<\infty$. Let $e_n=n^{-s}$ and $\widetilde{e}_n=e_n/\sqrt{w_n}$, where, we recall
    \[
w_n=\int_0^{\infty}n^{-2\sigma}d\mu(\sigma).
    \]    
     Then, the set $\{\widetilde{e}_n\}$ is an orthonormal basis of $\mathcal{A}^2_{\mu}$. Let $p\in[2,\infty)$. Then, by  \cite[Theorem 1.33]{zhu},
    \begin{equation*}
        \|T_g\|_{S_p}^p\geq\sum_{n\geq n_0}\|T_g(\widetilde{e}_n)\|_{\mathcal{A}_{\mu}^2}^p.
    \end{equation*}
    A simple computation yields
    \begin{align*}
        \|T_g(\widetilde{e}_n)\|_{\mathcal{A}^2_{\mu}}^2
        =
        \sum_{m\geq2}\left|\frac{\log m}{\log n+\log m}   \right|^2|a_m|^2\frac{w_{nm}}{w_n}
        &\geq
        \frac{(\log n_0)^2}{(\log n+\log n_0)^2}|a_{n_0}|^2\frac{w_{nn_0}}{w_n}\\
        &\geq  \frac{(\log n_0)^2}{4(\log n)^2}|a_{n_0}|^2w_{n_0}
    \end{align*}
    where in the last inequality we used Lemma \ref{proppesos}. Then,
    \[
     \|T_g\|_{S_p}^p\geq C \sum_{n\geq n_0}\frac{1}{(\log n)^p}=\infty.
    \]
    Therefore, the operator $T_g$ fails to be in any Schatten class $S_p$ for $p\in[2,\infty)$. The inclusion between the Schatten classes gives the conclusion for all finite $p$.
\end{proof}

\subsection{Compactness}
In this section we give a sufficient condition for compactness on $\mathcal{A}^p_{\mu}$ when $1<p<\infty$.
\begin{lema}\label{compact}
   Let $\{\mu_n\}$ be a sequence of positive real numbers and $X$ be a Banach space with a weak Schauder  basis $\{e_n\}$, namely,  
    \[
\|\sum_{n=1}^{N}a_ne_n\|_X\leq\mu_N\|\sum_{n=1}^{\infty}a_ne_n\|_X.
    \]
    Consider $\{\lambda_n\}$ a sequence of real numbers $\lambda_n\to0$ such that
    \begin{equation}\label{condo}
        \sum_{n=1}^{\infty}\mu_n|\lambda_n-\lambda_{n+1}|<\infty.
    \end{equation}
    Then, the operator
    \[
    T(\sum_{n=1}^{\infty}a_nn^{-s})=\sum_{n=1}^{\infty}a_n\lambda_ne_n
    \]
    is compact on $X$.
\end{lema}
\begin{proof}
Let $f(s)=\sum_{n\geq1}a_ne_n$ be a norm-$1$ vector in $X$. For each $N\in\N$, we consider the finite rank operator $T_N$ given by
\[
T_N(f)=\sum_{n=1}^Na_n\lambda_ne_n.
\]
We also let $S_n=\sum_{n=1}^na_ke_k$. Using Abel's summation formula and condition \eqref{condo}, we find that
\begin{align*}
    \|Tf-T_Nf\|_X=\|\sum_{n=N+1}^{\infty}a_n\lambda_ne_{n}\|_X
    &=
    \|\sum_{n=N+1}^{\infty}\lambda_n(S_n-S_{n-1})\|_X\\
    &=
\|\sum_{n=N+1}^{\infty}S_n(\lambda_n-\lambda_{n+1})\|_X
+\lambda_{N+1}\|S_N\|_X
\\
    &
    \lesssim \|f\|_X
    \sum_{n=N+1}^{\infty}\mu_n|\lambda_n-\lambda_{n+1}|.
\end{align*}
Taking the supremum over the norm-$1$ Dirichlet series and using the fact that the series in \eqref{condo} converges, we get  $\Vert T-T_N\Vert\to 0$ and the conclusion follows.
\end{proof}

\begin{prop}\label{policomp}
  Let $1<p<\infty$ and $g$ be a Dirichlet polynomial. Then, the operator $T_g$ is compact on $\mathcal{A}^p_{\mu}$.
\end{prop}
\begin{proof}
Fix $p\in(1,\infty)$. Since $T_g$ is linear on $g$ and the finite sum of compact operators is compact, it suffices to prove the result for the Dirichlet monomials $e_k=k^{-s}$, $k\geq2$. First, notice that, for $k$ fixed, a direct computation shows that
\begin{align*}
    T_{e_k}(e_n)=w_ne_{kn},
\end{align*}
where $w_n=\frac{\log k}{\log k+\log n}\cdot$ Hence, for $f(s)=\sum_{n\geq2}a_ne_n\in\mathcal{A}^p_{\mu}$, we have that
\[
T_{e_k}(f)=T_{e_k}(\sum_{n\geq2}a_ne_n)=
\sum_{n=2}^{\infty}a_nw_ne_{kn}
\]
Notice that the sequence $\lambda_n=w_n$ decreases to $0$ as $n\to\infty$, $k$ fixed. Now, since $\{e_n\}$ is a Schauder basis on $\mathcal{A}^p_{\mu}$ (\cite{aleolsen}, \cite{pascal}), then condition \eqref{condo} is satisfied with $\mu_n=c$ since $\{\lambda_n\}$ is decreasing. Note that since $M_{e_k}$ is a contraction on $\mathcal{A}^p_{\mu}$, multiplying by $e_k$ does not affect the conclusion. Hence, we can apply Lemma \ref{compact} to
the operator $T_{e_k}$, obtaining its compactness, as desired.
\end{proof}
\begin{teor}
 Let $1<p<\infty$. If $g$ is a Dirichlet series belonging to the space $\emph{Bloch}_{\mu,0}(\C_+)$ and $\mu$ satisfies the $H_1$-condition, then $T_g$ is compact on $\mathcal{A}^p_{\mu}$.
        \end{teor}
\begin{proof}
Let $g$ be as above and take $\varepsilon>0$. Then, Proposition \ref{densite} guarantees the existence of a Dirichlet polynomial $P$ such that $\|g-P\|_{\text{Bloch}_{\mu}(\C_+)}<\varepsilon$. Now, thanks to Theorem \ref{carlesonberg} we know that
\[
\|T_P-T_g\|\lesssim\|g-P\|_{\text{Bloch}_{\mu}(\C+)}<\varepsilon.
\]
Using Proposition \ref{policomp} we obtain the desired result.
\end{proof}

\begin{remark}
    The same type of argument works in $\mathcal{H}^p$ for $1<p<\infty$ if we change $\text{Bloch}_{\mu,0}(\C_+)$ by $\text{VMOA}(\C_+)$ (see \cite[Section 7]{brevig} for the definition).
\end{remark}

\section{A radicality result}
The following radicality result was proven in \cite{alem}. 
\begin{teor}\label{radicality}
Let $f$ be an analytic function in the unit disc $\D$ such that $f^n\in\mathcal{B}$ for some $n\in\N$. Then, $f^m\in\mathcal{B}$ for all $m<n$, $m\in\N$.
\end{teor}
The proof of the theorem depends strongly on having a good definition of the norm of the Bloch functions in terms of some Hilbert space norm. In the case of the Bloch space this definition is done through Garsia's norm, see \cite[Theorem 1]{axler}. More specifically, the following characterisation of the Bloch space $\mathcal{B}$ of the unit disc is used. As usual, for $a\in\D$ fixed,
\[
\varphi_a(z)=\frac{a-z}{1-\overline{a}z}.
\]

\begin{teor}
Let $f$ be an analytic function in the unit disc $\D$. Then, $f$ belongs to the Bloch space in the unit disc $\mathcal{B}$ if, and only if, 
\begin{equation*}
    \sup_{a\in\D}\|f\circ\varphi_{a}-f(a)\|_{A^2}<\infty,
\end{equation*}
moreover, with equivalent norms.
\end{teor}
Here $A^2=A^2(\D)$ stands for the Bergman space of unit disc $\D$, namely, a holomorphic function $f$ in the unit disc belongs to this space $A^2$ iff 
\begin{equation*}
    \|f\|_{A^2}^2=\frac{1}{\pi}\int_{\D}|f(z)|^2dA(z)<\infty.
\end{equation*}
Let $d\in\N$. From now on, $\D^d$ will denote the polydisc, that is,  %the denote to the polydisc, \textcolor{red}{that is}this is, 
we consider $d$ copies of the unit disc $\D$.

\ 
The purpose of this section is to prove that Theorem \ref{radicality} still holds for the Bloch space in the polydisc $\D^d$, and then to apply that fact to Dirichlet series, under the form
\begin{teor}\label{th:rad} Let $f(s)=\sum_{k=1}^\infty a_k k^{-s}\in \mathcal{D}_d$, depending only on $d$ given primes.  If $f^n\in \mathcal{B}$, then $f^m\in \mathcal{B}$ for all integers $1\leq m\leq n$.
\end{teor} 

Let us first define this space:
\begin{equation*}
    \mathcal{B}(\D^d)=\{
f:\D^d\to\C \ \text{holomorphic}: \max_{1\leq j\leq d}\sup_{z\in\D^d}|D_{j}f(z)|(1-|z_j|^2)<\infty
    \}.
\end{equation*}
Therefore, we want to prove the following $d$-dimensional version of Theorem \ref{radicality}.
\begin{teor}\label{dradicality}
    Let $f$ be an analytic function in $\D^d$ such that $f^n\in\mathcal{B}(\D^d)$ for some $n\in \N$. Then, $f^m\in\mathcal{B}(\D^d)$ for all $m\leq n$, $m\in\N$.
\end{teor}

Before proving the theorem, we will need to establish a analogue of Garsia's norm for the polydisc setting. In a similar fashion as in the unit disc case, we let $A^2(\D^d)$ be the Hilbert space of holomorphic functions in $\D^d$ such that
\[
\frac{1}{\pi}\int_{\D^d}|f(z)|^2dA_d(z)<\infty.
\]
Similarly, we need to state a result regarding the automorphisms of the polydisc $\D^d$. See \cite[Lemma 3.4]{queffi}.
\begin{lema}\label{automorf}
The analytic map $\Phi:\D^d\to\D^d$ belongs to \emph{Aut}$(\D^d)$ if, and only if,
\begin{equation*}
    \Phi(z)
    =
    \left(
\Phi_1(z),\ldots,
\Phi_d(z)
    \right),\quad \text{where 
 }\Phi_j(z)=\varepsilon_j\frac{z_{\sigma(j)}-a_j}{1-\overline{a_j}z_{\sigma(j)}},
\end{equation*}
for some permutation $\sigma$ of $\{1,\ldots,d\}$, $(a_1,\ldots,a_d)\in\D^d$ and complex signs $\varepsilon_1,\ldots,\varepsilon_d$.
\end{lema}

\begin{lema}
    Let $f$ be an analytic function in $\D^d$ such that $f\in\mathcal{B}(\D^d)$ and let $\Phi:\D^d\to\D^d$ belong to $\emph{Aut}(\D^d)$. Then, $f\circ\Phi\in\mathcal{B}(\D^d)$.
\end{lema}
\begin{proof}
Notice that to prove the result it suffices to establish the following 
\begin{equation*}
    |D_{j}(f\circ\Phi)(z)|(1-|z_j|^2)
    =
    |D_{j}f(\Phi(z))|(1-|\Phi_j(z)|^2).
\end{equation*}
Applying the derivation operator in the left hand-side
\begin{equation*}
   |D_{j}(f\circ\Phi)(z)|(1-|z_j|^2)
   =
    |D_{j}(f\circ\Phi)(z)||D_{j}\Phi(z)|(1-|z_j|^2).
\end{equation*}
The result follows from the following observation, in which we make use of Lemma \ref{automorf}
\begin{align*}
    |\Phi'_j(z_{\sigma(j)})|(1-|z_j|^2)=\frac{(1-|a_j|^2)(1-|z_{\sigma(j)}|^2)}{|1-\overline{a_j}z_{\sigma(j)}|^2}
    =
    1-\frac{|a_j-z_{\sigma(j)}|^2}{|1-\overline{a_j}z_{\sigma(j)}|^2}
    =1-|\Phi_j(z)|^2.
\end{align*}
Therefore,
\begin{align*}
    |D_{j}(f\circ\Phi)(z)|(1-|z_j|^2)
    &=
    |D_{j}(f)(\Phi(z))||\Phi'_j(z)|(1-|z_j|^2)\\
    &=|D_{j}(f)(\Phi(z))|(1-|\Phi_j(z)|^2).
\end{align*}
\end{proof}

\begin{lema}\label{invari}
Let $f$ be an analytic function in the polydisc $\D^d$. Then, $f$ belongs to the Bloch space $\mathcal{B}(\D^d)$ if and only if, 
\begin{equation*}
    \sup_{a\in\D^d}\|f\circ\Phi_a-f(a)\|_{A^2(\D^d)}<\infty.
\end{equation*}
\end{lema}
\begin{proof}
We begin by showing the necessity. Hence, let $F(t)=f(zt)-f(0)$, with $0\leq t\leq1$ and $z\in\D^d$. Then,
\begin{align*}
    |f(z)-f(0)|=
    \left|\int_0^1
    F'(t)dt
    \right|
    &\leq\sum_{j=1}^d\int_0^1|z_j||D_jf(tz)|dt\\
    &
    \leq C
\sum_{j=1}^d\int_0^1|z_j|(1-t|z_j|)^{-1}
    \|f\|_{\mathcal{B}(\D^d)}dt\\
    &\leq -C\|f\|_{\mathcal{B}(\D^d)}\ln(1-|z|),
\end{align*}
where $(1-|z|)=(1-|z_1|)\cdots(1-|z_d|)$. Let $a\in\D^d$. Now, thanks the previous estimate, we have that (with $C>0$)
\begin{align*}
    |f\circ\Phi_a(z)-f(a)|\leq -C\|f\|_{\mathcal{B}(\D^d)}\ln((1-|z|)).
\end{align*}
Eventually, taking the $A^2(\D^d)$ norm in the latter inequality, we are led to
\begin{align*}
    \|f\circ\Phi_a-f(a)\|^2_{A^2(\D^d)}\leq
    C\|f\|_{\mathcal{B}^2(\D^d)}\int_{\D^d}|\ln(1-|z|)|^2dA_d(z)
\end{align*}
a simple computation allows us to check that the right-hand side integral is finite and the necessity follows.

\ 
Regarding the sufficiency, we use the fact that
\[
|Df_j(0)|\leq 2\|f\|_{A^2(\D^d)}, \quad 1\leq j\leq d.
\]
Therefore, for each fixed $a\in\D^d$, we replace $f$ with $f\circ\Phi_a-\Phi(a)$, so that
\[
(1-|a_j|^2)|D_jf(a)|\leq 2\|f\circ\Phi_a-\Phi(a)\|_{A^2(\D^d)},
\]
and the claim follows.
\end{proof}
Once we have these lemmas, the proof of Theorem \ref{dradicality} follows using an analogous argument to the one displayed in \cite{alem}, based on Hilbert spaces techniques. We recall it for the sake of completeness.
\begin{proof}[Proof of Theorem \ref{dradicality}]
In order to ease the reading of the proof, we shall write note $\Vert.\Vert_B$ for the Bloch norm of the Bloch space $\mathcal{B}(\D^d)$ in the polydisc and $\Vert .\Vert$ for the Bergman norm of the Bergman space $A^2(\D^d)$ of the polydisc. We will prove that, if $m\leq n$, we have that 
 $$(\Vert f^m\Vert_B)^{1/m} \leq (\Vert f^n\Vert_B)^{1/n}.$$
 Thanks to Lemma \ref{invari}, it suffices to prove that, for every $a\in \D^d$ and every analytic function $f$, we have that
 \begin{equation}\label{suf}
     (\Vert f^m\circ\Phi_a -f^{m}(a)\Vert)^{1/m} \leq (\Vert f^n\circ\Phi_a -f^{n}(a)\Vert)^{1/n}.
 \end{equation}
We can suppose that $a=0$. We shall also use Pythagora's theorem in the Hilbert space $A^2(\D^d)$:
\begin{equation} \label{pyth}  (\Vert f^k -f^{k}(0)\Vert)^{2}=\Vert f^k\Vert^2-|f(0)|^{2k}.\end{equation} 
Using this with $k=m$ and H\"older's inequality
\begin{align*}
    (\Vert f^n -f^{n}(0)\Vert)^{2}=\Vert f^n\Vert^2-|f(0)|^{2n}&\geq \Vert f^m\Vert^{2n/m}-|f(0)|^{2n}\\&=\big(\Vert f^m-f^{m}(0)\Vert^2+|f(0)|^{2m}\big)^{n/m}-|f(0)|^{2n}.
\end{align*}
We conclude using now the fact
$(x+y)^\alpha\geq x^\alpha+y^\alpha$ if $x,y\geq 0$ and $\alpha\geq 1$, for $\alpha=n/m$, which gives us that
$$(\Vert f^n -f^{n}(0)\Vert)^{2}\geq (\Vert f^m -f^{m}(0)\Vert)^{2n/m}$$
which gives us \eqref{suf} for $a=0$, which suffices to obtain the desired conclusion.

\end{proof}
\begin{proof}[Proof of Theorem \ref{th:rad}]
If $f^n$ belongs to $\text{Bloch}(\C_+)\cap\mathcal{D}_d$, then its Bohr lift $\mathcal{B}(f^n)$ belongs to $\text{Bloch}(\D^d)$ and the conclusion follows by an application of Theorem \ref{dradicality} and coming back to the space $\text{Bloch}(\C_+)\cap\mathcal{D}_d$, taking the inverse Bohr lift $\mathcal{B}^{-1}$.
\end{proof}

\end{document}